\documentclass[12pt]{article}
\title{Finite burden in multivalued algebraically closed fields}
\author{Will Johnson}

\usepackage{amsmath, amssymb, amsthm}    	
\usepackage{fullpage} 	
\usepackage{amscd}
\usepackage{hyperref}
\usepackage[all]{xy}
\usepackage[T1]{fontenc}
\usepackage{centernot}

\DeclareMathOperator*{\forkindep}{\raise0.2ex\hbox{\ooalign{\hidewidth$\vert$\hidewidth\cr\raise-0.9ex\hbox{$\smile$}}}}

\newcommand{\Abs}{\operatorname{Abs}}

\newcommand{\diag}{\operatorname{diag}}
\newcommand{\qftp}{\operatorname{qftp}}

\newcommand{\Img}{\operatorname{Im}}

\newcommand{\Gal}{\operatorname{Gal}}

\newcommand{\Frac}{\operatorname{Frac}}

\newcommand{\characteristic}{\operatorname{char}}
\newcommand{\Spec}{\operatorname{Spec}}

\newcommand{\res}{\operatorname{res}}
\newcommand{\Aut}{\operatorname{Aut}}

\newcommand{\acl}{\operatorname{acl}}
\newcommand{\dcl}{\operatorname{dcl}}
\newcommand{\tp}{\operatorname{tp}}

\newcommand{\Val}{\operatorname{Val}}

\newtheorem{theorem}{Theorem}[section] 
\newtheorem{lemma}[theorem]{Lemma}
\newtheorem{claim}[theorem]{Claim}
\newtheorem{definition}[theorem]{Definition}
\newtheorem{corollary}[theorem]{Corollary}

\newtheorem{fact}[theorem]{Fact}

\newtheorem{conjecture}[theorem]{Conjecture}

\newtheorem{remark}[theorem]{Remark}
\newtheorem{example}[theorem]{Example}
\newtheorem{proposition}[theorem]{Proposition}

\newtheorem*{theorem-star}{Theorem}
\newtheorem*{conjecture-star}{Conjecture}

\newcommand{\Aa}{\mathbb{A}}
\newcommand{\Qq}{\mathbb{Q}}

\newcommand{\Rr}{\mathbb{R}}

\newcommand{\Zz}{\mathbb{Z}}
\newcommand{\Nn}{\mathbb{N}}
\newcommand{\Cc}{\mathbb{C}}

\newcommand{\Mm}{\mathbb{M}}
\newcommand{\Ff}{\mathbb{F}}
\newcommand{\Pp}{\mathbb{P}}

\newcommand{\Oo}{\mathcal{O}}
\newcommand{\mm}{\mathfrak{m}}
\newcommand{\pp}{\mathfrak{p}}
\newcommand{\nn}{\mathfrak{n}}

\begin{document}
\maketitle

\begin{abstract}
  We prove that an expansion of an algebraically closed field by $n$
  arbitrary valuation rings is NTP${}_2$, and in fact has finite
  burden.  It fails to be NIP, however, unless the valuation rings
  form a chain.  Moreover, the incomplete theory of algebraically
  closed fields with $n$ valuation rings is decidable.
\end{abstract}

\section{Introduction}
Fix an integer $n$.  Consider the theory $ACv^nF$ of 1-sorted
structures $(K,+,\cdot,\Oo_1,\ldots,\Oo_n)$, where $(K,+,\cdot)
\models ACF$ and each $\Oo_i$ is (a unary predicate for) a valuation
ring on $K$.

Our main results are as follows:
\begin{enumerate}
\item The (incomplete) theory $ACv^nF$ is decidable: there is an
  algorithm which inputs $\varphi$ and outputs whether $ACv^nF \vdash
  \varphi$.
\item If $M \models T$, then $M$ has finite burden, hence is strong,
  NTP${}_2$.
\item If $M \models T$, then $M$ is NIP if and only if the valuation
  rings are pairwise comparable.
\end{enumerate}
Chapter 11 of \cite{myself} considered the more restrictive class of
structures in which the $\Oo_i$ are non-trivial and independent.  The
resulting theory turns out to be the model companion of the theory of
fields with $n$ valuation rings.  In this paper, we generalize the
results of \cite{myself} by eliminating the assumptions of
independence and non-triviality.

Rather than working directly with models of $ACv^nF$, it is more
convenient to work with certain definitional expansions which are
better behaved---for example, they are model complete.  We briefly summarize the situation.

By a \emph{finite tree}, we shall mean a finite poset $(P,\le)$
containing a minimal element $\bot$, such that every interval
$[\bot,p]$ is a chain.  A \emph{branch} of $P$ is a subposet of the
form $\{x \in P : x \ge a\}$ where $a$ is a minimal element of $P
\setminus \{\bot\}$.  The tree $P$ can be written as a disjoint union
\begin{equation*}
  P = \{\bot\} \cup P_1 \cup \cdots \cup P_n
\end{equation*}
where $P_1, \ldots, P_n$ are the distinct branches of $P$.  Each
branch $P_i$ is itself a finite tree.

To any finite tree $P$, we shall associate a theory $T_P$.  A model of
$T_P$ is an algebraically closed field $(K,+,\cdot,\Oo_p : p \in P)$
with a valuation ring $\Oo_p$ for each $p \in P$, satisfying some
axioms.  The important properties are
\begin{enumerate}
\item If $P_1, \ldots, P_n$ are the branches of $P$, then a model of
  $T_P$ is essentially an algebraically closed field $K$ with $n$
  independent non-trivial valuations $\Oo_1,\ldots,\Oo_n$, and a
  $T_{P_i}$-structure on the residue field of $\Oo_i$.
\item Every algebraically closed multivalued field
  $(K,+,\cdot,\Oo_1,\ldots,\Oo_n)$ admits a definitional expansion to
  a model of $T_P$, essentially by adding unary predicates for the
  joins $\Oo_i \cdot \Oo_j$.
\end{enumerate}
The first point allows us to mimic the arguments used in the case of
independent valuations.  The second point relates the theories $T_P$
and $ACv^nF$.

The paper is outlined as follows.  In \S\ref{res-struct}, we consider
the general setting of multi-valued fields with residue structure, and
derive a relative model completeness result in the case of independent
non-trivial valuations.  Essentially, we prove the following: if $T_1,
\ldots, T_n$ are model-complete theories expanding ACF, then model
completeness holds in the theory of algebraically closed fields with
$n$ independent non-trivial valuation rings $\Oo_1,\ldots,\Oo_n$ with
$T_i$-structure on the residue field of $\Oo_i$.  See
Lemma~\ref{mildly-annoying}.

In \S\ref{cast-of-characters} we introduce the aforementioned theory
$T_P$ and apply the results of \S\ref{res-struct} to prove that $T_P$
is model complete.  Moreover, we show that $T_P$ is the model
companion of a simpler theory $T_P^0$.  (Models of $T_P^0$ are exactly
the subfields of models of $T_P$.)

In \S\ref{amalg-in-tp}, we prove that $T_P^0$ has the amalgamation
property over algebraically closed bases.  From this, we deduce
several consequences, such as the usual criterion for elementary
equivalence: two models $M_1, M_2$ of $T_P$ are elementarily
equivalent iff the substructures $\Abs(M_1)$ and $\Abs(M_2)$ are
isomorphic.  This in turn yields decidability of $T_P$.  The proof of
amalgamation in $T_P^0$ relies on an amalgamation lemma in ACVF, which
we prove in \S\ref{controlled-amalg}.  The lemma says that when
amalgamating valued fields, we have complete freedom in how we
amalgamate the residue fields.

The rest of the paper is devoted to the classification-theoretic
dividing lines NTP${}_2$ and NIP.  In \S\ref{probable-truth} we define
a canonical Keisler measure on the set of complete types extending any
quantifier-free type.  More precisely, given any model $K \models
T_P^0$, we define a Keisler measure on the space of completions of
$T_P \cup \diag(K)$.  (This is a variant of the Keisler measure
defined in \S 11.4 of \cite{myself}.)  Some of the key properties of
the Keisler measure rely on an analysis in \S\ref{cubes-galore} of
extensions of nested valuation rings in certain diagrams of fields.
The analysis is notationally confusing, but not deep.

In \S\ref{finite-burden}, we
verify that models of $T_P$ have finite burden, using a minor lemma
proven in \S\ref{sec:coheirs}.  In \S\ref{sec:nip} we turn to the
matter of NIP, reviewing the argument from \cite{myself} \S 11.5.1
that algebraically closed fields with independent valuations cannot be
NIP.  We conclude in \S\ref{sec:horizons} by discussing different
directions in which the results can probably be generalized.

\subsection{Notation}
We will generally use the letter $\Oo$ for valuation rings, $\mm$ for
their maximal ideals, and lowercase roman letters (such as $k, \ell$)
for residue fields.

If $K$ is a valued field with valuation ring $\Oo$, we let $\res \Oo$
denote the residue field.  We also write $\res K$ for the residue
field, if $\Oo$ is clear from context.  We will also use $\res(x)$ to
denote the residue of $x$.

When multiple valuation rings $\Oo_1, \ldots, \Oo_n$ are in play, we
will use subscripts to indicate which residue map we are talking
about: $\res_i(x)$ denotes the residue of $x$ in the $i$th residue
field $k_i = \res \Oo_i$.

If $\Oo' \subseteq \Oo$ are two valuation rings on a field $K$, we let
$\Oo' \div \Oo$ denote the unique valuation ring on $\res \Oo$ whose
composition with $\Oo$ is $\Oo'$.

Two non-trivial valuation rings $\Oo_1, \Oo_2$ are \emph{independent}
if they induce distinct topologies.  An equivalent condition is that
$\Oo_1 \cdot \Oo_2 = K$.  Here, $\Oo_1 \cdot \Oo_2$ denotes the
join---the smallest valuation ring on $K$ containing both $\Oo_1$ and
$\Oo_2$.  It happens to agree with the setwise product \[\Oo_1 \cdot
\Oo_2 = \{x \cdot y : x \in \Oo_1 \text{ and } y \in \Oo_2\}.\]

We let $\Val(K)$ denote the poset of all valuation rings on $K$.  If
$\Oo \in \Val(K)$, we let $\Val(K|\Oo)$ denote the subset
\begin{equation*}
  \Val(K|\Oo) := \{\Oo' \in \Val(K) : \Oo' \subseteq \Oo\}.
\end{equation*}
The poset $\Val(K|\Oo)$ is canonically isomorphic to $\Val(\res \Oo)$
via the map
\begin{align*}
  \Val(K|\Oo) & \to \Val(\res \Oo) \\
  \Oo' & \mapsto \Oo' \div \Oo.
\end{align*}

\section{Multi-valued fields with residue structure} \label{res-struct}
If $X$ is a set and $\mathcal{T}_1, \ldots, \mathcal{T}_n$ are
topologies on $X$, we say that $\mathcal{T}_1,\ldots,\mathcal{T}_n$
are \emph{jointly independent} if the diagonal embedding
\begin{equation*}
  X \hookrightarrow (X,\mathcal{T}_1) \times \cdots \times (X,\mathcal{T}_n)
\end{equation*}
has dense image.  In other words, if $U_i$ is a non-empty
$\mathcal{T}_i$-open for each $i$, then $\bigcap_{i = 1}^n U_i$ is
non-empty.
\begin{fact}\label{strong-approx}
  Let $K$ be an algebraically closed field.  Let $\Oo_1, \ldots,
  \Oo_n$ be pairwise independent non-trivial valuation rings on $K$.
  Let $V \subseteq \Aa^n_K$ be an irreducible affine variety.  Then
  the metric topologies on $V$ are jointly independent.
\end{fact}
This should be a classical result, but I had trouble finding the
original reference; a later proof is in \cite{myself} Theorem 11.3.1.
\begin{lemma}\label{3-sorted-proto-consequence}
  Let $(K,\Oo) \le (K',\Oo')$ be an extension of models of ACVF.  Let
  $k \le k'$ be the corresponding residue field extension.  Then any
  $K$-definable subset of $(k')^n$ is quantifier-free $k$-definable in
  the pure ring structure on $k'$.
\end{lemma}
\begin{proof}
  This follows from the 3-sorted quantifier elimination in ACVF.  We
  shall also see another proof later (Remark~\ref{2-or-3}).
\end{proof}
\begin{corollary}\label{3-sorted-consequence}
  Let $(K,\Oo) \le (K',\Oo')$ be an extension of models of ACVF, and
  $k \le k'$ be the residue field extension.  Let $V \subseteq
  \Aa^{n+m}$ be a quasi-affine variety over $K$.  There is a
  quantifier-free formula $R(\vec{\xi})$ in the language of rings over
  $k$ such that
  \begin{align*}
    R(K) &= \{\res(\vec{x}) : (\vec{x},\vec{y}) \in V(K)\} \\
    R(K') &= \{\res(\vec{x}) : (\vec{x},\vec{y}) \in V(K')\}
  \end{align*}
  where $\vec{x}, \vec{\xi}$ are $n$-tuples, $\vec{y}$ is an
  $m$-tuple, and $\res(\vec{x})$ is understood componentwise.
\end{corollary}
\begin{proof}
  This follows from Lemma~\ref{3-sorted-proto-consequence}; the same
  formula $R$ works for both $(K,\Oo)$ and $(K',\Oo')$ by model
  completeness of ACVF.
\end{proof}

In this section, we shall always consider $n$-fold multivalued fields
in the $(n+1)$-sorted language $(K,k_1,\ldots,k_n)$ with
\begin{itemize}
\item Field structure on each $K$ and $k_i$.
\item Unary predicates $\Oo_1, \ldots, \Oo_n$ for the valuation rings.
\item Partial maps $\res_i : K \rightsquigarrow k_i$.
\end{itemize}

\begin{remark}\label{ec-comments}
  Let $(K,k_1,\ldots,k_n)$ be an $n$-fold multivalued field, and let
  $(K',k'_1,\ldots,k'_n)$ be an extension.  Suppose $K$ is e.c.\ in
  $K'$.
  \begin{enumerate}
  \item If $K'$ is an algebraically closed field, then $K$ is an
    algebraically closed field.
  \item If $\Oo'_i$ is non-trivial, then $\Oo_i$ is non-trivial.
  \item If $\Oo'_i$ and $\Oo'_j$ are independent, then $\Oo_i$ and
    $\Oo_j$ are independent.
  \end{enumerate}
\end{remark}
\begin{proof}
  \begin{enumerate}
  \item If $K'$ is algebraically closed and $K$ is not, take a monic
    polynomial $P(X) \in K[X]$ without a solution in $K$.  Then
    \begin{equation*}
      \exists x : P(x) = 0
    \end{equation*}
    holds in $K'$, but not in $K$, contrary to existential closedness.
  \item Suppose $\Oo'_i$ is non-trivial but $\Oo_i$ is trivial.  Then
    \begin{equation*}
      \exists x : x \notin \Oo_i
    \end{equation*}
    holds in $K'$ but not in $K$.
  \item Suppose, say, $\Oo_1$ and $\Oo_2$ fail to be independent.
    Then $\Oo_0 = \Oo_1 \cdot \Oo_2$ is a non-trivial valuation ring.
    Let $\mm_i$ denote the maximal ideal of $\Oo_i$.  Then
    \begin{align*}
      \mm_0 \subset \mm_1 & \subset \Oo_1 \subset \Oo_0 \\
      \mm_0 \subset \mm_2 & \subset \Oo_2 \subset \Oo_0.
    \end{align*}
    Because $\Oo_0$ is non-trivial, there is some non-zero
    $\varepsilon \in \mm_0$.  Then
    \begin{align*}
       \varepsilon \Oo_1 \subseteq \varepsilon \Oo_0 & \subseteq \mm_0 \\
       \varepsilon \Oo_2 \subseteq \varepsilon \Oo_0 & \subseteq \mm_0.
    \end{align*}
    So the existential statement
    \begin{equation*}
      \exists x, y : x \in \Oo_1 \wedge y \in \Oo_2 \wedge 1 +
      \varepsilon x = \varepsilon y
    \end{equation*}
    is false in $K$, as $\mm_0$ and $1 + \mm_0$ are disjoint.  On the
    other hand, this existential statement is true in $K'$ by
    approximation (Fact~\ref{strong-approx}) on the line $\{(x,y) : 1
    + \varepsilon x = \varepsilon y\}$.
  \end{enumerate}
\end{proof}

\begin{lemma}\label{boring-extension}
  Let $(K,\Oo_1,\ldots,\Oo_n)$ be a field with $n$ valuations.  Then
  we can embed $(K,\Oo_1,\ldots,\Oo_n)$ into a larger $n$-valued field
  $(K',\Oo'_1,\ldots,\Oo'_n)$ such that
  \begin{enumerate}
  \item $K'$ is algebraically closed.
  \item Each $\Oo'_i$ is non-trivial.
  \item The $\Oo'_i$ are pairwise independent.
  \end{enumerate}
\end{lemma}
\begin{proof}
  The class of $n$-fold multivalued fields is an $\forall
  \exists$-elementary class, so we may assume $(K,\Oo_1,\ldots,\Oo_n)$
  is e.c.\ among fields with $n$-valuations.  We claim that $K$ has the
  desired properties of $K'$.  By Remark~\ref{ec-comments}, it
  suffices to produce extensions of $K$ having each of the properties
  separately.
  \begin{enumerate}
  \item By the Chevalley Extension Theorem (\cite{PE} Theorem 3.1.1), we can extend
    each $\Oo_i$ to a valuation ring $\Oo'_i$ on $K^{alg}$.  Then
    $(K^{alg},\Oo'_1,\ldots,\Oo'_n)$ is an extension in which $K'$ is
    algebraically closed.
  \item Suppose, say, $\Oo_1$ is trivial.  Let $K((T))$ be the Laurent
    field extension, and let $\Oo'_1$ be the discrete valuation ring
    $K[[T]]$.  Then $\Oo'_1$ extends the trivial valuation $\Oo_1$ on
    $K$.  For $i \ne 1$, let $\Oo'_i$ be an arbitrary extension of
    $\Oo_i$ to $K((T))$.  Then $(K((T)),\Oo'_1,\ldots,\Oo'_n)$ is an
    extension in which $\Oo'_1$ is non-trivial.
  \item Suppose, say, $\Oo_1$ and $\Oo_2$ fail to be independent.  Let
    $\Oo'_1$ be the valuation on $K(T)$ obtained by composing the
    $T$-adic valuation with $\Oo_1$.  Let $\Oo'_2$ be the valuation on
    $K(T)$ obtained by composing the $(T+1)$-adic valuation with
    $\Oo_2$.  Then $\Oo'_i$ extends $\Oo_i$ for $i = 1, 2$, and
    $\Oo'_1$ is independent from $\Oo'_2$, because the $T$-adic and
    $(T+1)$-adic valuations on $K(T)$ are independent.  For $i \ne 1,
    2$ choose $\Oo'_i$ to be an arbitrary extension of $\Oo_i$.
  \end{enumerate}
\end{proof}

\begin{remark}\label{basic-fact}
  If $T$ is a model complete theory and $M \le N \models T$, then $M$
  is not e.c.\ in $N$ unless $M \models T$.
\end{remark}
\begin{proof}
  Suppose $M$ is e.c.\ in $N$.  We claim that $M \preceq N$ by the
  Tarski-Vaught test.  Let $X \subseteq N$ be a non-empty
  $M$-definable set in the structure $N$.  By model completeness, $X =
  \pi(Y)$ where $\pi : N^n \twoheadrightarrow N$ is a coordinate
  projection and $Y \subseteq N^n$ is quantifier-free definable over
  $M$.  Non-emptiness of $X$ implies non-emptiness of $Y$.  As $M$ is
  e.c.\ in $N$, the set $Y \cap M^n$ is non-empty.  Therefore $\pi(Y
  \cap M^n) \subseteq X \cap M$ is non-empty.
\end{proof}

\begin{lemma}\label{mildly-annoying}
  Let $T_1, \ldots, T_n$ be model-complete expansions of ACF.  Let $T$
  be the theory of $(n+1)$-sorted structures $(K,k_1,\ldots,k_n)$ with
  field structure on $K$, residue maps $\res_i : K \rightsquigarrow
  k_i$ for each $i$, and with $(T_i)_\forall$ structure on each $k_i$.
  Then $(K,k_1,\ldots,k_n)$ is existentially closed if and only if
  \begin{enumerate}
  \item $K = K^{alg}$,
  \item each valuation ring $\Oo_i$ is non-trivial
  \item the valuation rings $\Oo_i$ are pairwise independent, and
  \item each $k_i$ is a model of $T_i$.
  \end{enumerate}
  In particular, $T$ has a model companion.
\end{lemma}
\begin{proof}
  We first show the necessity of the listed conditions.  Suppose
  $(K,k_1,\ldots,k_n)$ is existentially closed.  By
  Lemma~\ref{boring-extension} we can embed $(K,k_1,\ldots,k_n)$ into
  a larger multivalued field $(K',k'_1,\ldots,k'_n)$, without residue
  structure, such that $K'$ is algebraically closed, each $\Oo'_i$ is
  non-trivial, and the $\Oo'_i$ are pairwise independent.

  We can take $(K',\ldots)$ to be highly saturated.  Then $k'_i$ is an
  algebraically closed field of high transcendence degree over $k$.
  As $k_i \models (T_i)_\forall$ we can find a $T_i$-structure on
  $k'_i$ extending the $(T_i)_\forall$-structure on $k_i$.  This
  endows $(K',k'_1,\ldots,k'_n)$ with a $T$-structure, such that $k'_i
  \models T_i$.

  Now $(K,k_1,\ldots,k_n)$ is existentially closed in the extension
  $(K',k'_1,\ldots,k'_n)$.  Then $k_i$ is e.c.\ in $k'_i$, so $k_i
  \models T_i$ by Remark~\ref{basic-fact}.  Similarly, by
  Remark~\ref{ec-comments}, the $\Oo_i$ are non-trivial and
  independent, and $K = K^{alg}$.  Thus $(K,k_1,\ldots,k_n)$ must
  satisfy the listed conditions if it is existentially closed.

  Next, suppose that $(K,k_1,\ldots,k_n)$ satisfies all the listed
  conditions.  Let $(K',k'_1,\ldots,k'_n)$ be an extension; we must
  show that $K$ is e.c.\ in $K'$.  Enlarging $K'$, we may assume that
  $K'$ is e.c., hence satisfies the listed conditions.  Suppose we are
  given some existential formula over $K$ which is true in $K'$; we
  must show it is true in $K$.  By adding dummy variables, we reduce
  to an existential formula of the form
  \begin{align*}
    \exists \vec{x}_1, \ldots, \vec{x}_n, \vec{y}, \vec{\xi}_1,
    \ldots, \vec{\xi}_n :& R_0(\vec{x}_1,\ldots,\vec{x}_n,\vec{y}) \\
    \wedge & \bigwedge_{i = 1}^n \left(\res_i(\vec{x}_i) =
    \vec{\xi}_n\right) \\ \wedge & \bigwedge_{i = 1}^n R_i(\vec{\xi}_i),
  \end{align*}
  where
  \begin{itemize}
  \item $\vec{x}_1, \ldots, \vec{x}_n, \vec{y}$ are tuples from the
    big field sort
  \item $\vec{\xi}_i$ is a tuple from the $i$th residue field sort,
  \item $\res_i(-)$ acts on vectors componentwise
  \item $R_0$ is a quantifier-free formula over $K$ in the pure field
    language
  \item $R_i$ is a quantifier-free formula over $k_i$ in the language
    of $T_i$.
  \end{itemize}
  The relation $R_0$ can be written as a disjunction of conjunctions;
  we may restrict to one of the conjunctions, reducing to the case
  where the existential formula has the form
  \begin{align*}
    \exists \vec{x}_1, \ldots, \vec{x}_n, \vec{y}, \vec{\xi}_1,
    \ldots, \vec{\xi}_n :&  \bigwedge_{j = 1}^m
    \left(P_i(\vec{x}_1,\ldots,\vec{x}_n,\vec{y}) = 0\right) \\ \wedge &
    \left(Q(\vec{x}_1,\ldots,\vec{x}_n,\vec{y}) \ne 0\right) \\ \wedge & \bigwedge_{i = 1}^n
    \left(\res_i(\vec{x}_i) = \vec{\xi}_n\right) \\ \wedge & \bigwedge_{i =
      1}^n R_i(\vec{\xi}_i),
  \end{align*}
  where the $P_i$ and $Q$ are polynomials over $K$.

  Fix some tuple
  $(\vec{a}'_1,\ldots,\vec{a}'_n,\vec{b}',\vec{\alpha}'_1,\ldots,\vec{\alpha}'_n)$
  in the big model $K'$ witnessing the existential statement.  Adding
  more $P_i$, we may assume that the $P_i$ cut out an irreducible
  affine variety $V$ over $K$, namely the locus of
  $(\vec{a}'_1,\ldots,\vec{a}'_n,\vec{b}')$ over $K$.  Let $V\setminus W$
  be the Zariski open subset of $V$ cut out by $Q \ne 0$.

  By Corollary~\ref{3-sorted-consequence}, we can find quantifier-free
  $\mathcal{L}_{rings}$ formulas $R'_i(\vec{\xi}_i)$ such that in both
  $K$ and $K'$,
  \begin{equation*}
    R'_i(\vec{\xi}_i) \iff \exists
    (\vec{x}_1,\ldots,\vec{x}_n,\vec{y}) \in V \setminus W :
    \res(\vec{x}_i) = \vec{\xi}_i.
  \end{equation*}
  In particular, $R'_i(\vec{\alpha}'_i)$ holds.  Replacing
  $R_i(\vec{\xi}_i)$ with $R_i(\vec{\xi}_i) \wedge R'_i(\vec{\xi}_i)$,
  we may assume that
  \begin{equation}
    R_i(\vec{\xi}_i) \implies \exists (\vec{x}_1,\ldots,\vec{x}_n,\vec{y})
    \in V \setminus W : \res(\vec{x}_i) = \vec{\xi}_i.
    \label{well}
  \end{equation}
  Each residue field $k_i$ is a model of $T_i$, hence existentially
  closed in $k'_i$.  Therefore we can find
  $\vec{\alpha}_1,\ldots,\vec{\alpha}_n$ in $k_1,\ldots,k_n$ such
  that $R_i(\vec{\alpha}_i)$ holds.  Let $X_i$ be
  \begin{equation*}
    X_i := \{(\vec{x}_1,\ldots,\vec{x}_n,\vec{y}) \in V(K) :
    \res_i(\vec{x}_i) = \vec{\alpha}_i.\}
  \end{equation*}
  Each $X_i \setminus W$ is non-empty by (\ref{well}) and choice of
  $\vec{\alpha}_i$.  Moreover, each $X_i \setminus W$ is an
  $\Oo_i$-adically open subset of $V(K)$.  By independence, the
  intersection $\bigcap_i (X_i \setminus W)$ is non-empty.  Let
  $(\vec{a}_1,\ldots,\vec{a}_n,\vec{b})$ be a point in the
  intersection.  Then
  \begin{enumerate}
  \item $(\vec{a}_1,\ldots,\vec{a}_n,\vec{b})$ lies on $V(K)
    \setminus W$.
  \item $\res_i(\vec{a}_i) = \vec{\alpha}_i$ because
    $(\vec{a}_1,\ldots,\vec{a}_n,\vec{b}) \in X_i$.
  \item $R_i(\vec{\alpha}_i)$ holds for each $i$, by choice of
    $\vec{\alpha}_i$.
  \end{enumerate}
  Therefore, the existential statement holds in $K$, witnessed by
  $(\vec{a}_1,\ldots,\vec{a}_n,\vec{b},\vec{\alpha}_1,\ldots,\vec{\alpha}_n)$.
  So $K$ is existentially closed.
\end{proof}

\section{The theories $T_P$ and $T_P^0$} \label{cast-of-characters}
\begin{definition}
  A \emph{tree} is a $\wedge$-semilattice $P$ with bottom element
  $\bot$ such that for every $x \in P$, the interval $[\bot,x]$ is
  totally ordered.  A \emph{homomorphism} of trees is a
  $\wedge$-semilattice homomorphism mapping $\bot$ to $\bot$.
\end{definition}
The set $\Val(K)$ of valuation rings on a field $K$ is naturally a
tree, after reversing the order.  The join operation is
\begin{equation*}
  \Oo_1 \vee \Oo_2 = \Oo_1 \cdot \Oo_2.
\end{equation*}

Fix a finite tree $P$.
\begin{definition}
  The theory $T_P$ is the theory of algebraically closed fields $K$
  with injective tree homomorphisms
  \begin{align*}
    P & \hookrightarrow \Val(K)^{op} \\
    p & \mapsto \mathcal{O}_p.
  \end{align*}
\end{definition}
In other words, a model of $T_P$ is a structure $(K,\mathcal{O}_p : p
\in P)$ where
\begin{itemize}
\item $K$ is an algebraically closed field.
\item $\Oo_p$ is a valuation ring for each $p \in P$.
\item If $p < p'$, then $\Oo_p \supsetneq \Oo_{p'}$.
\item $\Oo_\bot = K$.
\item For any $p, p' \in P$
  \begin{equation*}
    \Oo_{p \wedge p'} = \Oo_p \cdot \Oo_{p'}.
  \end{equation*}
\end{itemize}
\begin{example}\label{flat-example}
  Let $P$ be the flat tree $\{1, \ldots, n, \bot\}$ in which
  \begin{equation*}
    a \wedge b = 
    \begin{cases}
      \bot & \text{ if } a \ne b \\
      a & \text{ if } a = b.
    \end{cases}
  \end{equation*}
  Then a model of $T_P$ is essentially a structure
  $(K,\Oo_1,\ldots,\Oo_n)$ where $K$ is an algebraically closed field
  and the $\Oo_i$ are pairwise independent non-trivial valuation
  rings.
\end{example}
\begin{example} \label{acvnf-tp}
  Let $K$ be an algebraically closed field and let $\Oo_1, \ldots,
  \Oo_n$ be finitely many arbitrary valuation rings on $K$.  Let $P$
  be the (finite) sub-$\wedge$-semilattice of $Val(K)^{op}$ generated
  by $\{\Oo_1, \ldots, \Oo_n, K\}$.  Then $P$ is a tree, and
  $(K,\Oo_1,\ldots,\Oo_n)$ is bi-interpretable with the model
  $(K,\Oo_p : p \in P) \models T_P$, where $p \mapsto \Oo_p$ is the
  tautological map.
\end{example}
\begin{definition}
  Let $P$ be a finite tree.  Let $T^0_P$ be the theory
  whose models are structures $(K,\Oo_p : p \in P)$ where $K$ is a
  field and
  \begin{equation*}
    p \mapsto \Oo_p
  \end{equation*}
  is a weakly order-preserving map $P \to Val(K)^{op}$ sending $\bot$
  to $K$.
\end{definition}
In other words, a model of $T^0_P$ is a structure $(K,\Oo_p : p \in
P)$ where
\begin{itemize}
\item $K$ is a field (not necessarily algebraically closed).
\item $\Oo_p$ is a valuation ring for each $p \in P$.
\item If $p < p'$, then $\Oo_p \supseteq \Oo_{p'}$ (but the inclusion needn't be strict).
\item $\Oo_\bot = K$.
\end{itemize}
Note that for any $p, p' \in P$,
\begin{equation*}
  \Oo_{p \wedge p'} \supseteq \Oo_p \cdot \Oo_{p'}
\end{equation*}
but equality needn't hold.
\begin{example}
  If $P$ is the tree of Example~\ref{flat-example}, then a
  model of $T^0_P$ is a field $K$ with $n$ valuation rings on it.
\end{example}

\begin{theorem}
  $T_P$ is the model companion of $T_P^0$.
\end{theorem}
\begin{proof}
  Let $a_1, \ldots, a_n$ enumerate the minimal elements of $P
  \setminus \{\bot\}$.  Let $P_i$ be the subposet $\{x \in P : x \ge
  a_i\}$.  Note that $P_i$ is a finite tree with bottom
  element $a_i$.  By induction, $T_{P_i}^0$ is the model companion of
  $T_{P_i}$.  Let $T$ be the theory of $(n+1)$-sorted structures
  $(K,k_1,\ldots,k_n)$ with field structure on $K$, residue maps
  $\res_i : K \rightsquigarrow k_i$ for each $i$, and with
  $T_{P_i}^0$-structure on each $k_i$.

  Given a model $(K,k_1,\ldots,k_n)$ of $T$, we get a model $(K,
  \Oo^K_p : p \in P)$ of $T_P^0$ by defining
  \begin{itemize}
  \item $\Oo_\bot$ to be $K$.
  \item $\Oo_p$ to be the composition of $K \rightsquigarrow k_i$ with
    $\Oo^{k_i}_p$, if $p \ge a_i$.
  \end{itemize}
  This gives an equivalence of categories from the category of models
  of $T$ (with morphisms the \emph{embeddings}) to the category of of
  models of $T_P^0$ (with morphisms the embeddings).  Moreover, this
  equivalence of categories sends elementary embeddings to elementary
  embeddings in both directions.
  
  By Lemma~\ref{mildly-annoying}, $T$ has a model companion $T'$ whose
  models are characterized by the following additional axioms:
  \begin{enumerate}
  \item $K$ is algebraically closed.
  \item Each valuation ring $\Oo_i$ is non-trivial.
  \item The valuation rings $\Oo_i$ are pairwise independent.
  \item Each residue field $k_i$ is a model of $T_{P_i}$.
  \end{enumerate}
  Under the equivalence of categories, models of $T'$ correspond to
  models of $T_P$.  Therefore, $T_P$ is the model companion of
  $T_P^0$.
\end{proof}

\section{Controlled amalgamation in ACVF} \label{controlled-amalg}
\begin{definition}
  A ring homomorphism $f : R \to K$ to a field $K$ is \emph{dominant}
  if $K$ is generated as a field by $\Img(f)$.
\end{definition}
For a fixed ring $R$, dominant morphisms out of $R$ are classified up
to equivalence by prime ideals of $R$.

By \emph{the category of fields} we mean the full subcategory of the
category of rings.  Note that homomorphisms are embeddings.
\begin{definition}
  Let
  \begin{equation*}
    \xymatrix{ F \ar[r] \ar[d] & K_1 \\ K_2 & }
  \end{equation*}
  be a diagram in the category of fields.
  \begin{itemize}
  \item An \emph{amalgamation} of $K_1$ and $K_2$ over $F$ is a
    diagram
    \begin{equation*}
      \xymatrix{F \ar[r] \ar[d] & K_1 \ar[d] \\ K_2 \ar[r] & L}
    \end{equation*}
    extending the given diagram.  When the maps $K_i \to L$ are clear,
    one says that \emph{$L$ is an amalgamation of $K_1$ and $K_2$ over
      $F$}.
  \item Two amalgamations $L$ and $L'$ are equivalent if there is an
    isomorphism $L \to L'$ such that
    \begin{equation*}
      \xymatrix{ K_1 \ar[r] \ar[d] & L' \\ L \ar[ur] & K_2 \ar[l] \ar[u]}
    \end{equation*}
    commutes.
  \item An amalgamation $L$ is \emph{reduced} if $L$ is the compositum
    $K_1'K_2'$, where $K_i'$ is the image of $K_i \to L$.
    Equivalently, $L$ is reduced if the morphism
    \begin{equation*}
      K_1 \otimes_F K_2 \to L
    \end{equation*}
    is dominant.
  \item The \emph{reduction} of an amalgamation $L$ is the subfield
    $K_1'K_2'$ of $L$, where $K_i'$ is the image of $K_i \to L$.
  \item An \emph{amalgamation type} is an equivalence class of reduced
    amalgamations, or equivalently, a prime ideal in $K_1 \otimes_F
    K_2$.
  \item \emph{The amalgamation type} of an amalgamation $L$ is the
    equivalence class of the reduction, or equivalently, $\ker K_1
    \otimes_F K_2 \to L$.
  \item If $K_1 \otimes_F K_2$ is a domain, the \emph{independent
    amalgamation type} is the amalgamation type corresponding to the
    zero ideal $(0) \le K_1 \otimes_F K_2$, and an \emph{independent
      amalgamation} is one of independent type.
  \end{itemize}
\end{definition}

\begin{lemma}\label{a-fine-lemma}
  Let
  \begin{equation*}
    \xymatrix{ K_0 \ar[r] \ar[d] & K_1 \\ K_2 & }
  \end{equation*}
  be a diagram of embeddings of valued fields $(K_i,\Oo_i)$.  Then the
  natural ring homomorphism
  \begin{equation*}
    \Oo_1 \otimes_{\Oo_0} \Oo_2 \to K_1 \otimes_{K_0} K_2
  \end{equation*}
  is injective.
\end{lemma}
\begin{proof}
  Because $\Oo_1$ is torsionless as
  an $\Oo_0$-module, it is flat.  Therefore, the natural map
  \begin{equation*}
    \Oo_1 \otimes_{\Oo_0} \Oo_2 \to \Oo_1 \otimes_{\Oo_0} K_2
  \end{equation*}
  is an injection.  Similarly, $K_2$ is a flat $\Oo_0$-module, so
  \begin{equation*}
    \Oo_1 \otimes_{\Oo_0} K_2 \to K_1 \otimes_{\Oo_0} K_2
  \end{equation*}
  is injective.  Finally, the map
  \begin{equation*}
    K_1 \otimes_{\Oo_0} K_2 \to K_1 \otimes_{K_0} K_2
  \end{equation*}
  is an isomorphism because $\Oo_0 \to K_0$ is a (category-theoretic)
  epimorphism and tensor products are pushouts in the category of
  rings.
\end{proof}
\begin{remark}\label{injtastic}
  If $f : A \hookrightarrow B$ is an injective homomorphism of rings,
  then every minimal prime of $A$ extends to a prime of $B$.  Indeed,
  if $\pp$ is a minimal prime of $A$, let $S = A \setminus \pp$.
  Injectivity of $f$ implies that $f(S)$ is a multiplicative subset of
  $B$ not containing zero, so the localization $f(S)^{-1}B$ is
  non-trivial.  Any prime ideal of $f(S)^{-1}B$ pulls back to a prime
  in $A$ contained in $\pp$, hence equal to $\pp$ by minimality.
\end{remark}
\begin{remark}\label{local-hom}
  The category of valued fields and embeddings is equivalent to the
  category of valuation rings and injective local homomorphisms (i.e.,
  injective ring homomorphisms $f : \Oo_1 \to \Oo_2$ such that
  $f^{-1}(\mm_2) = \mm_1$).
\end{remark}

\begin{lemma}\label{amalgamation-master}
  Let
  \begin{equation*}
    \xymatrix{ K_0 \ar[r] \ar[d] & K_1 \\ K_2 & }
  \end{equation*}
  be a diagram of embeddings of valued fields.  Let $\Oo_i$ and $k_i$
  be the valuation ring and residue field of $K_i$.  Given any
  amalgamation type $\tau$ of $k_1$ and $k_2$ over $k_0$, there
  exists an amalgamation of valued fields:
  \begin{equation*}
    \xymatrix { K_0 \ar[r] \ar[d] & K_1 \ar[d] \\ K_2 \ar[r] & L}
  \end{equation*}
  such that
  \begin{enumerate}
  \item If $\ell$ denotes the residue field of $L$, then the
    amalgamation type of $\ell$ over $k_1$ and $k_2$ is $\tau$.
  \item If $K_1 \otimes_{K_0} K_2$ is a domain, then $L$ is an
    independent amalgamation.
  \item $L$ is a reduced amalgamation of $K_1$ and $K_2$.
  \end{enumerate}
\end{lemma}
\begin{proof}
  Requirement (3) is trivial to arrange, by replacing $L$ with its
  subfield generated by $K_1$ and $K_2$.  So we will forget
  requirement (3).

  Let $\nn$ be the prime ideal of $k_1 \otimes_{k_0} k_2$ associated
  to the amalgamation type $\tau$.  Let $\pp_1$ be the pullback of
  $\nn$ under the surjective ring homomorphism
  \begin{equation*}
    \Oo_1 \otimes_{\Oo_0} \Oo_2 \twoheadrightarrow k_1 \otimes_{k_0} k_2.
  \end{equation*}
  Consider the commutative diagram of sets for $i = 1, 2$:
  \begin{equation*}
    \xymatrix{ \Spec k_i \ar[d] & \Spec k_1 \otimes_{k_0} k_2 \ar[l] \ar[d] \\
      \Spec \Oo_i & \Spec \Oo_1 \otimes_{\Oo_0} \Oo_2. \ar[l]}
  \end{equation*}
  There is only one point in the top left set, and it maps to the
  maximal ideal $\mm_i \in \Spec \Oo_i$.  Since $\pp_1$ comes from
  $\nn$ in the top left, the restriction of $\pp_1$ to $\Oo_i$ must be
  $\mm_i$.

  Now let $\pp_0$ be some minimal prime in $\Oo_1 \otimes_{\Oo_0}
  \Oo_2$, chosen to lie below $\pp_1$.  Consider the commutative
  diagram
  \begin{equation*}
    \xymatrix{ \Spec K_i \ar[d] & \Spec K_1 \otimes_{K_0} K_2 \ar[l] \ar[d] \\
      \Spec \Oo_i & \Spec \Oo_1 \otimes_{\Oo_0} \Oo_2. \ar[l]}
  \end{equation*}
  By Lemma~\ref{a-fine-lemma} and Remark~\ref{injtastic}, $\pp_0$
  comes from an element of the top right corner.  Again, $\Spec K_i$
  has only one point and it maps to the zero ideal in $\Spec \Oo_i$,
  so the restriction of $\pp_0$ to $\Oo_i$ must then be the zero
  ideal.

  By the Chevalley Extension Theorem (\cite{PE} Theorem 3.1.1) there
  is a valuation ring $(\Oo_3,\mm_3)$ and a homomorphism
  \begin{equation*}
    \Oo_1 \otimes_{\Oo_0} \Oo_2 \to \Oo_3
  \end{equation*}
  under which $\mm_3$ and $(0)$ pull back to $\pp_1$ and $\pp_0$,
  respectively.  This yields a diagram
  \begin{equation*}
    \xymatrix{ \Oo_0 \ar[r] \ar[d] & \Oo_1 \ar[d] \\ \Oo_2 \ar[r] & \Oo_3}.
  \end{equation*}
  Under the composition
  \begin{equation*}
    \Oo_1 \to \Oo_1 \otimes_{\Oo_0} \Oo_2 \to \Oo_3,
  \end{equation*}
  the prime ideals $\mm_3$ and $(0)$ pull back to $\pp_1$ and $\pp_0$,
  and then to $\mm_1$ and $(0)$, respectively.  It follows that $\Oo_1
  \to \Oo_3$ is an injective local homomorphism.  The same holds for
  $\Oo_2 \to \Oo_3$ similarly.  By Remark~\ref{local-hom}, we get a
  diagram of valued fields
  \begin{equation*}
    \xymatrix { K_0 \ar[r] \ar[d] & K_1 \ar[d] \\ K_2 \ar[r] & L}
  \end{equation*}
  From the induced diagram
  \begin{equation*}
    \xymatrix{\Oo_1 \otimes_{\Oo_0} \Oo_2 \ar[d] \ar[r] & k_1 \otimes_{k_0} k_2 \ar[d] \\
      \Oo_3 \ar[r] & \ell}
  \end{equation*}  
  we see that the kernel of $k_1 \otimes_{k_0} k_2 \to \ell$ pulls
  back to $\pp_1$ on $\Oo_1 \otimes_{\Oo_0} \Oo_2$.  So $\ker k_1
  \otimes_{k_0} k_2 \to \ell$ and $\nn$ have the same image under the
  injection $\Spec k_1 \otimes_{k_0} k_2 \hookrightarrow \Spec \Oo_1
  \otimes_{\Oo_0} \Oo_2$.  Therefore, the kernel equals $\nn$, so the
  amalgamation type of $\ell$ is $\tau$ as desired.

  Finally, suppose that $K_1 \otimes_{K_0} K_2$ is a domain.  By
  Lemma~\ref{a-fine-lemma}, the map $\Oo_1 \otimes_{\Oo_0} \Oo_2 \to
  K_1 \otimes_{K_0} K_2$ is injective, so $\Oo_1 \otimes_{\Oo_0}
  \Oo_2$ is a domain.  The minimal prime $\pp_0$ must then be the zero
  ideal $(0)$.  Now, in the diagram
  
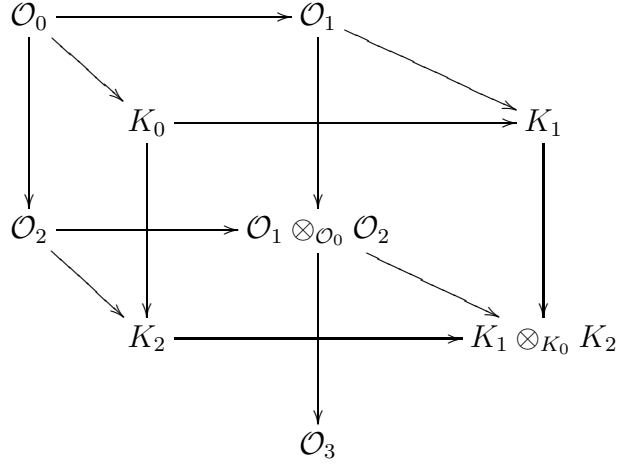
\begin{equation*}
    \xymatrix{\Oo_0 \ar[rr] \ar[dr] \ar[dd] & & \Oo_1 \ar[dd] \ar[dr] & \\
      & K_0 \ar[rr] \ar[dd] & & K_1 \ar[dd] \\
      \Oo_2 \ar[rr] \ar[dr] & & \Oo_1 \otimes_{\Oo_0} \Oo_2 \ar[dd] \ar[dr] & \\
      & K_2 \ar[rr] & & K_1 \otimes_{K_0} K_2 \\
      & & \Oo_3 & }
  \end{equation*}
  every ring is a domain, and moreover \emph{every} morphism is an
  injection:
  \begin{itemize}
  \item The maps $\Oo_i \to K_i$ are injections by definition.
  \item The map $\Oo_1 \otimes_{\Oo_0} \Oo_2 \to K_1 \otimes_{K_0}
    K_2$ is an injection by Lemma~\ref{a-fine-lemma}.
  \item The maps out of $K_0, K_1, K_2$ are injections because the
    $K_i$ are fields.
  \item The maps $\Oo_i \to \Oo_1 \otimes_{\Oo_0} \Oo_2$ are injective
    because the diagram commutes, and the other path from $\Oo_i$ to
    $K_1 \otimes_{K_0} K_2$ is made of injections.
  \item The map $\Oo_1 \otimes_{\Oo_0} \Oo_2 \to \Oo_3$ is injective
    because the pullback of $(0)$ under this map was $\pp_0 = (0)$, by
    choice of $\Oo_3$.
  \end{itemize}
  Therefore, the above diagram belongs to the category of domains and
  embeddings.  Applying the functor $\Frac(-)$ yields the diagram
  
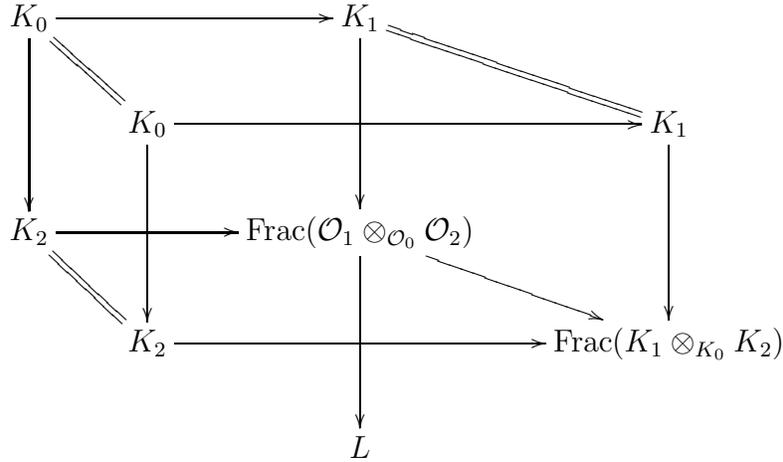
\begin{equation*}
    \xymatrix{K_0 \ar[rr] \ar@{=}[dr] \ar[dd] & & K_1 \ar[dd] \ar@{=}[dr] & \\
      & K_0 \ar[rr] \ar[dd] & & K_1 \ar[dd] \\
      K_2 \ar[rr] \ar@{=}[dr] & & \Frac(\Oo_1 \otimes_{\Oo_0} \Oo_2) \ar[dd] \ar[dr] & \\
      & K_2 \ar[rr] & & \Frac(K_1 \otimes_{K_0} K_2) \\
      & & L & }
  \end{equation*}
  This diagram contains three amalgamations of $K_1$ and $K_2$ over
  $K_0$, namely $L$, $\Frac(\Oo_1 \otimes_{\Oo_0} \Oo_2)$ and
  $\Frac(K_1 \otimes_{K_0} K_2)$.  Moreover, the diagram shows that
  they have the same amalgamation type.  By definition, $\Frac(K_1
  \otimes_{K_0} K_2)$ has the independent amalgamation type.  Thus $L$
  is also an independent amalgamation.
\end{proof}
\begin{remark}
  Lemma~\ref{amalgamation-master} can be used to prove quantifier
  elimination in ACVF.  The Lemma implies that the class of valued
  fields has the amalgamation property.  By abstract nonsense, it only
  remains to prove that models of ACVF are 1-e.c., in other words,
  \begin{equation*}
    M_1 \models \exists x : \varphi(x;\vec{b}) \implies M_2 \models
    \exists x : \varphi(x;\vec{b})
  \end{equation*}
  for any extension $M_1 \le M_2$ of models, tuple $\vec{b}$ from
  $M_1$, and quantifier-free formula $\varphi(x;\vec{y})$ \emph{with
    $x$ a singleton}.  The 1-e.c.\ property can be verified in a
  straightforward fashion from the swiss cheese decomposition of
  \emph{quantifier-free definable} sets.
\end{remark}
\begin{remark}
  Lemma~\ref{amalgamation-master} implies amalgamation for the class
  of two-sorted structures $(K,k)$ where $K$ is a valued field and $k$
  is an \emph{extension} of the residue field.  Indeed, suppose we are given a
  diagram
  \begin{equation*}
    \xymatrix{ (K_0,k_0) \ar[r] \ar[d] & (K_1,k_1) \\ (K_2,k_2) & }
  \end{equation*}
  of embeddings of such structures.  First, amalgamate $k_1$ and $k_2$
  over $k_0$ into a monster model $\Mm$ of ACF:
  \begin{equation*}
    \xymatrix{ k_0 \ar[r] \ar[d] & k_1 \ar[d] \\ k_2 \ar[r] & \Mm.}
  \end{equation*}
  This induces an amalgamation of the residue fields:
  \begin{equation} \label{tau-1}
    \xymatrix{ \res K_0 \ar[r] \ar[d] & \res K_1 \ar[d] \\ \res K_2 \ar[r] & \Mm.}
  \end{equation}
  By Lemma~\ref{amalgamation-master} one can then amalgamate the
  valued fields
  \begin{equation*}
    \xymatrix{ K_0 \ar[r] \ar[d] & K_1 \ar[d] \\ K_2 \ar[r] & L}
  \end{equation*}
  in such a way that
  \begin{equation*}
    \xymatrix{ \res K_0 \ar[r] \ar[d] & \res K_1 \ar[d] \\ \res K_2 \ar[r] & \res L}
  \end{equation*}
  has the same amalgamation type as (\ref{tau-1}).  Because the
  amalgamation types agree, there is an embedding of $\res L$ into
  $\Mm$ such that the diagram commutes:
  \begin{equation*}
    \xymatrix{ \res K_0 \ar[rr] \ar[dr] \ar[dd] & & \res K_1 \ar[dd] \ar[dr] \\
      & k_0 \ar[rr] \ar[dd] & & k_1 \ar[dd] \\
      \res K_2 \ar[dr] \ar[rr] & & \res L \ar[dr] \\
      & k_2 \ar[rr] & & \Mm.}
  \end{equation*}
  The embedding of $\res L$ into $\Mm$ yields a structure $(L,\Mm)$
  and the above diagram means that
  \begin{equation*}
    \xymatrix{ (K_0,k_0) \ar[r] \ar[d] & (K_1,k_1) \ar[d] \\ (K_2,k_2) \ar[r] & (L,\Mm)}
  \end{equation*}
  commutes.
\end{remark}
\begin{remark}\label{2-or-3}
  Amalgamation for the 2-sorted structures $(K,k)$ implies a sort of
  quantifier elimination for 2-sorted ACVF.  Specifically, if $(K,k)$
  is one of these two-sorted structures (possible with $k$ strictly
  greater than $\res K$) then any two embeddings of $(K,k)$ into a
  model of ACVF have the same type.  This implies
  Lemma~\ref{3-sorted-proto-consequence}.
\end{remark}
\begin{lemma}\label{amalg-2}
  Let $K_0 = K_0^{alg}$ and let
  \begin{equation*}
    \xymatrix { K_0 \ar[r] \ar[d] & K_1 \ar[d] \\ K_2 \ar[r] & K_3}
  \end{equation*}
  be an independent amalgamation of fields.  Let $\Oo_1, \Oo_2$ be
  valuation rings on $K_1, K_2$ having the same restriction $\Oo_0$ to
  $K_0$.  Then there is a valuation ring $\Oo_3$ on $K_3$ extending
  $\Oo_1$ and $\Oo_2$ such that the induced amalgamation of residue
  fields
  \begin{equation*}
    \xymatrix { k_0 \ar[r] \ar[d] & k_1 \ar[d] \\ k_2 \ar[r] & k_3}
  \end{equation*}
  is independent.
\end{lemma}
\begin{proof}
  Because $K_0$ is algebraically closed, so is $k_0$.  Therefore, $k_1
  \otimes_{k_0} k_2$ is a domain, so it makes sense to talk about the
  independent amalgamation type on the residue fields.  By
  Lemma~\ref{amalgamation-master} there is some amalgamation of valued
  fields
  \begin{equation*}
    \xymatrix { K_0 \ar[r] \ar[d] & K_1 \ar[d] \\ K_2 \ar[r] & L}
  \end{equation*}
  such that
  \begin{itemize}
  \item $L$ is an independent amalgamation of $K_1$ and $K_2$ over $K_0$.
  \item $\res L$ is an independent amalgamation of $k_1$ and $k_2$ over $k_0$.
  \item $L$ is a reduced amalgamation of $K_1$ and $K_2$ over $K_0$.
  \end{itemize}
  By assumption, $K_3$ also has the independent amalgamation type, so
  its reduction is isomorphic to $L$.  Therefore, there is an
  embedding of $L$ into $K_3$ such that the following diagram of pure
  fields commutes:
  \begin{equation*}
    \xymatrix{ K_0 \ar[r] \ar[d] & K_1 \ar[d] \ar[ddr] & \\
      K_2 \ar[r] \ar[rrd] & L \ar[dr] & \\
      & & K_3.}
  \end{equation*}
  Let $\Oo_3$ be any valuation ring on $K_3$ extending the valuation
  ring on $L$.  Then the above diagram becomes a diagram of valued
  fields.  Moreover, the induced diagram of residue fields looks like
  \begin{equation*}
    \xymatrix{ k_0 \ar[r] \ar[d] & k_1 \ar[d] \ar[ddr] & \\
      k_2 \ar[r] \ar[rrd] & \res L \ar[dr] & \\
      & & \res K_3.}
  \end{equation*}
  Thus the amalgamation type of $\res K_3$ over $k_1$ and $k_2$ is the
  same as $\res L$, namely the independent type.
\end{proof}

Recall that $\Val(K)$ denotes the poset of valuation rings on $K$, and
$\Val(K|\Oo)$ denotes the subposet of valuation rings below a given
$\Oo \in \Val(K)$.  If $\Oo' \in \Val(K|\Oo)$, then $\Oo' \div \Oo$
denotes the valuation ring on $\res \Oo$ whose composition with $\Oo$
is $\Oo'$.
\begin{remark}\label{meteor}
  ~
  \begin{enumerate}
  \item If $L/K$ is an extension of fields, there is a restriction map
    \begin{align*}
      \Val(L) \to \Val(K) \\
      \Oo \mapsto \Oo \cap K
    \end{align*}
  \item \label{other-restriction} If $L/K$ is an extension of fields,
    if $\Oo_L$ is a valuation ring on $L$, and if $\Oo_K = \Oo_L \cap
    K$ is the restriction to $K$, then there is a restriction map
    $\Val(L|\Oo_L) \to \Val(K|\Oo_K)$.  Moreover, the diagram commutes
    \begin{equation*}
      \xymatrix{ \Val(L|\Oo_L) \ar[r] \ar[d] & \Val(K|\Oo_K) \ar[d] \\
        \Val(L) \ar[r] & \Val(K)}
    \end{equation*}
    where the vertical maps are inclusions.
  \item\label{petunia} If $K$ is a valued field and $\Oo_K$ is a
    valuation ring on $K$ with residue field $k = \res \Oo_K$, then
    there is a bijection
    \begin{align*}
      \Val(K|\Oo_K) & \to \Val(k) \\
      \Oo & \mapsto \Oo \div \Oo_K.
    \end{align*}
  \item\label{orchid} If $L/K$ is an extension of fields, if $\Oo_L$
    is a valuation ring on $L$, if $\Oo_K = \Oo_L \cap K$ is the
    restriction to $K$, and if $\ell \to k$ is the residue field
    embedding $\res \Oo_L \to \res \Oo_K$, then the diagram
    \begin{equation*}
      \xymatrix{ \Val(L|\Oo_L) \ar[r] \ar[d] & \Val(K|\Oo_K) \ar[d] \\
        \Val(\ell) \ar[r] & \Val(k)}
    \end{equation*}
    commutes, where the vertical maps are the bijections of
    (\ref{petunia}) and the horizontal maps are the restriction maps
    of (\ref{other-restriction}).
  \end{enumerate}
\end{remark}

\begin{lemma}\label{comet}
  Let
  \begin{equation*}
    \xymatrix{(F,\Oo_F) \ar[r] \ar[d] & (K_1,\Oo_{K_1}) \ar[d]
      \\ (K_2,\Oo_{K_2}) \ar[r] & (L,\Oo_L)}
  \end{equation*}
  be a diagram of valued fields.  Suppose $F$ is algebraically closed
  and $\res \Oo_L$ is an independent amalgamation of $\res \Oo_{K_1}$
  and $\res \Oo_{K_2}$ over $\res \Oo_F$.  Given $\Oo'_F \subseteq
  \Oo_F$, $\Oo'_{K_1} \subseteq \Oo_{K_1}$, and $\Oo'_{K_2} \subseteq
  \Oo_{K_2}$ such that $\Oo'_{K_1}$ and $\Oo'_{K_2}$ restrict to
  $\Oo'_F$, there is some $\Oo'_L \subseteq \Oo_L$ extending
  $\Oo'_{K_1}$ and $\Oo'_{K_2}$ such that $\res \Oo'_L$ is an
  independent amalgamation of $\res \Oo'_{K_1}$ and $\res \Oo'_{K_2}$
  over $\res \Oo'_F$.
\end{lemma}
\begin{proof}
  This follows from Remark~\ref{meteor}.\ref{orchid} and
  Lemma~\ref{amalg-2}.  Specifically, Remark~\ref{meteor}.\ref{orchid}
  shows the commutativity of the diagram
  \begin{equation*}
    \xymatrix{
      \Val(F|\Oo_F) \ar[dd]    & & \Val(K_1|\Oo_{K_1}) \ar[dd] \ar[ll] & \\
      & \Val(K_2|\Oo_{K_2}) \ar[ul] \ar[dd]   & & \Val(L|\Oo_L) \ar[dd] \ar[ul] \ar[ll]  \\
      \Val(\res \Oo_F)   & & \Val(\res \Oo_{K_1}) \ar[ll]  & \\
      & \Val(\res \Oo_{K_2}) \ar[ul]  & & \Val(\res \Oo_L) \ar[ll] \ar[ul] ,}
  \end{equation*}
  where the vertical maps are bijections.  The problem we need to
  solve is on the upper plane, but the diagram allows us to move the
  problem to the lower plane.  One concludes by applying
  Lemma~\ref{amalg-2} to the diagram
  \begin{equation*}
    \xymatrix{\res \Oo_F \ar[r] \ar[d] & \res \Oo_{K_1} \ar[d]
      \\ \res \Oo_{K_2} \ar[r] & \res \Oo_L.}
  \end{equation*}
\end{proof}

\section{Amalgamation in $T_P$} \label{amalg-in-tp}

\begin{proposition}
  Let $P$ be a finite tree.  Let
  \begin{equation*}
    \xymatrix{ K_0 \ar[r] \ar[d] & K_1 \\ K_2 & }
  \end{equation*}
  be a diagram of embeddings of models of $T_P^0$, with $K_0 =
  K_0^{alg}$.  Then the diagram can be completed to a diagram
  \begin{equation*}
    \xymatrix { K_0 \ar[r] \ar[d] & K_1 \ar[d] \\ K_2 \ar[r] & K_3}
  \end{equation*}
  of embeddings of models of $T_P^0$.  Furthermore, $K_3$ can be
  chosen to be an independent amalgamation of $K_1$ and $K_2$.
\end{proposition}
\begin{proof}
  Let $K_3 = \Frac(K_1 \otimes_{K_0} K_2)$.  Let $\Oo^p_i$ and $k^p_i$
  denote the $p$th valuation ring and residue field on $K_i$.  One
  chooses $\Oo^p_3$ on $K_3$ by upwards recursion on $p$, ensuring
  that $\res \Oo^p_3$ is an independent amalgamation of $\res \Oo^p_1$ and
  $\res \Oo^p_2$ over $\res \Oo^p_0$ at each step.  This is possible
  by Lemma~\ref{comet}.
\end{proof}
\begin{corollary}
  In $T_P$, field-theoretic algebraic closure agrees with
  model-theoretic algebraic closure.
\end{corollary}
\begin{proof}
  Suppose $M \models T_P$ and $K = K^{alg} \le M$.  Suppose $a \in
  \acl(K)$.  We claim $a \in K$.  Take a second copy $M'$ of $M$,
  amalgamated with $M$ independently over $K$ inside a third model
  $M'' \models T_P$.  By model completeness, $M' \preceq M'' \succeq
  M$.  Let $X$ be the $K$-definable finite set of conjugates of $a$.
  Then $a \in X(M) = X(M'') = X(M')$, so $a \in M \cap M'$.  In the
  ACF reduct,
  \begin{equation*}
    M' \forkindep_K M \implies a \forkindep_K a \implies a \in
    \acl(K) \implies a \in K.
  \end{equation*}
\end{proof}
\begin{corollary}
  Let $M_1, M_2$ be two models of $T_P$, let $K_i$ be an algebraically
  closed subfield of $M_i$, and $f : K_1 \to K_2$ be an isomorphism of
  $T_P^0$-structures.  Then $f$ is a partial elementary map.
\end{corollary}
\begin{proof}
  Amalgamate $M_1$ and $M_2$ over $K$ and use model completeness of
  $T_P$.
\end{proof}

\begin{definition}
  If $P$ is a finite tree, then $T^{alg}_P$ is $T^0_P$
  plus the axiom that $K \models ACF$.
\end{definition}
So $T^0_P \vdash T^{alg}_P \vdash T_P$.
\begin{corollary}
  $T_P$ is the model \emph{completion} of $T^{alg}_P$.
\end{corollary}
\begin{corollary}
  Let $M, M'$ be two models of $T_P$.  Then $M \equiv M'$ if and only
  if $\Abs(M) \cong \Abs(M')$, where $\Abs(M)$ denotes the
  substructure of ``absolute numbers,'' i.e., elements algebraic over
  the prime field.
\end{corollary}
The only valuation ring on $\Ff_p^{alg}$ is the trivial one, because
$(\Ff_P^{alg})^\times$ is torsion.  Therefore,
\begin{corollary}\label{pos-char-complete}
  If $M, M' \models T_P$ and $\characteristic(M) = \characteristic(M')
  > 0$, then $M \equiv M'$.
\end{corollary}

\begin{corollary} \label{almost-qe}
  Let $K$ be a model of $T_P^0$, let $\varphi(\vec{x})$ be a sentence
  in the language of $T_P$, and let $\vec{a}$ be a tuple from $K$.
  There is a finite normal extension $L/K$ such that for $M \models
  T_P$ extending $K$, whether or not $\varphi(\vec{a})$ holds in $M$
  is determined by the induced $T_P^0$ structure on (the copy of) $L$
  in $M$.
\end{corollary}

\begin{corollary}
  The (incomplete) theory $T_P$ is decidable: there is an algorithm
  which takes a sentence $\varphi$ and determines whether $T_P \vdash
  \varphi$.
\end{corollary}
\begin{proof}
  As $T_P$ is c.e., it suffices to show that the set of sentences
  consistent with $T_P$ is c.e.
  
  Let $\chi$ be a function from $P$ to $\{0,2,3,5,7,\ldots\}$
  satisfying the requirement that $\chi(x) = p \implies \chi(y) = p$
  for $x \le y \in P$ and $p \ne 0$.  For any such $\chi$, let
  $T_{P,\chi}$ be $T_P$ plus axioms asserting that
  $\characteristic(\res \Oo_x) = \chi(x)$ for all $x \in P$.  Define
  $T_{P,\chi}^0$ similarly.  Each $T_{P,\chi}$ is consistent, so it
  suffices to show that the set of sentences consistent with
  $T_{P,\chi}$ is c.e., uniformly in $\chi$.

  When $\chi(\bot) > 0$, the theory $T_{P,\chi}$ is complete by
  Corollary~\ref{pos-char-complete}, and therefore decidable.  So assume
  $\chi(\bot) = 0$.

  Let $a_1, \ldots, a_n$ enumerate the minimal $a \in P$ such that
  $\chi(a) > 0$, and let $p_i = \chi(a_i)$.
  \begin{claim}\label{middle-claim}
    Let $M_1, M_2$ be two models of $T_{P,\chi}^0$, algebraic over the
    prime field.  Let $f : M_1 \to M_2$ be an isomorphism of the
    underlying fields.  Then $f$ is an isomorphism of
    $T_{P,\chi}^0$-structures if and only if $f$ sends
    $\Oo^{M_1}_{a_i}$ to $\Oo^{M_2}_{a_i}$ for $i = 1, \ldots, n$.
  \end{claim}
  \begin{proof}
    Without loss of generality, $M_1$ and $M_2$ have the same
    underlying field $K$ and $f$ is the identity map $id_K : K \to K$.
    The ``only if'' direction is clear.  For the ``if'' direction,
    note that the non-trivial valuation rings on $K$ are pairwise
    incomparable and mixed characteristic, because $K$ is algebraic
    over the prime field.  Therefore, $\Oo_a$ must be trivial when
    $\chi(a) = 0$, and $\Oo_a$ must equal $\Oo_{a_i}$ when $a \ge
    a_i$.  So the $\Oo_{a_i}$ determine the other valuation rings.
  \end{proof}
  Let $\Psi$ be the set of
  sentences of the form
  \begin{equation*}
    \exists x : Q(x) = 0 \wedge R_1(x) \wedge \cdots \wedge R_n(x),
  \end{equation*}
  where $Q(X) \in \Qq[X]$ is a monic irreducible polynomial, $R_i(x)$
  is a quantifier-free predicate only involving the $a_i$th valuation
  ring, and $ACVF_{0,p_i} \vdash \exists x : Q(x) = 0 \wedge R_i(x)$.
  The set $\Psi$ is c.e., because the set of monic irreducible
  polynomials is c.e.
  \begin{claim}
    A sentence $\varphi$ is consistent with $T_{P,\chi}$ if and only
    if $T_{P,\chi} \cup \{\psi\} \vdash \varphi$ for some $\psi \in
    \Psi$.
  \end{claim}
  \begin{proof}
    For the ``if'' direction, we only need to show that the sentences
    \begin{equation*}
      \psi := \left(\exists x : Q(x) = 0 \wedge R_1(x) \wedge \cdots
      \wedge R_n(x)\right)
    \end{equation*}
    are consistent with $T_{P,\chi}$.  Fix a copy of $\Qq^{alg}$ and a
    root $\alpha$ of $Q(X)$.  For each $i$, we can find a valuation
    ring $\Oo_i$ on $\Qq^{alg}$ of mixed characteristic $(0,p_i)$,
    such that
    \begin{equation*}
      (\Qq^{alg},\Oo_i) \models R_i(\alpha).
    \end{equation*}
    Indeed, first choose an arbitrary valuation ring $\Oo'$ of mixed
    characteristic $(0,p_i)$, use the assumption on $R_i$ to find
    $\alpha' \in \Qq^{alg}$ such that
    \begin{equation*}
      (\Qq^{alg},\Oo') \models Q(\alpha') = 0 \wedge R_i(\alpha'),
    \end{equation*}
    and then move $\alpha'$ and $\Oo'$ to $\alpha$ and $\Oo_i$ by an
    automorphism in $\Gal(\Qq)$.  Now
    \begin{equation*}
      (\Qq^{alg},\Oo_1,\ldots,\Oo_n) \models \exists x : Q(x) = 0
      \wedge R_1(x) \wedge \cdots \wedge R_n(x),
    \end{equation*}
    witnessed by $\alpha$.  Expand $(\Qq^{alg},\Oo_1,\ldots,\Oo_n)$ to
    a model of $T_{P,\chi}^0$ as in the proof of
    Claim~\ref{middle-claim}, and then extend to a model $M \models
    T_{P,\chi}$.  Then $M \models \psi$.

    Conversely, suppose that $\phi$ holds in some model $M \models
    T_{P,\chi}$.  By Corollary~\ref{almost-qe}, there is a subfield $L
    \le M$ such that $[L : \Qq] < \infty$ and $T_{P,\chi} \cup
    \diag(L) \vdash \phi$.  Let $\alpha$ be a generator of $L =
    \Qq(\alpha)$.  Then
    \begin{equation*}
      T_{P,\chi} \cup \qftp(\alpha/\emptyset) \vdash \psi.
    \end{equation*}
    Let $\qftp_i(\alpha/\emptyset)$ be the quantifier-free type in the
    reduct $(M,\Oo_{a_i})$.  Then
    \begin{equation*}
      T_{P,\chi} \cup \bigcup_{i = 1}^n \qftp_i(\alpha/\emptyset)
      \vdash \qftp(\alpha/\emptyset)
    \end{equation*}
    essentially by Claim~\ref{middle-claim}.  By compactness and the
    lemma on constants, there are quantifier-free formulas $R_i(x) \in
    \qftp_i(\alpha/\emptyset)$ such that
    \begin{align*}
      T_{P,\chi} \cup \{\exists x : R_1(x) \wedge \cdots \wedge
      R_n(x)\} \vdash \varphi.
    \end{align*}
    Let $Q(X)$ be the minimal polynomial of $\alpha$ over $\Qq$ and let
    $\psi$ be the sentence
    \begin{equation*}
      \exists x : Q(x) = 0 \wedge R_1(x) \wedge \cdots \wedge R_n(x).
    \end{equation*}
    Then $T_{P,\chi} \cup \{\psi\} \vdash \varphi$ a fortiori.
    Moreover, for any $i$
    \begin{equation*}
      M \models \exists x : Q(x) = 0 \wedge R_i(x),
    \end{equation*}
    witnessed by $\alpha$.  So the formula $\exists x : Q(x) = 0
    \wedge R_i(x)$ is consistent with ACVF${}_{0,p_i}$.  But
    ACVF${}_{0,p_i}$ is complete, so ACVF${}_{0,p_i} \vdash \exists x
    : Q(x) = 0 \wedge R_i(x)$.  Therefore $\psi \in \Psi$.
  \end{proof}
  Given the claim, it follows that the set of sentences consistent
  with $T_{P,\chi}$ is c.e., uniformly in $\chi$.  Taking the union
  over all $\chi$, the set of sentences consistent with $T_P$ is c.e.
  The set of consequences of $T_P$ is trivially c.e., and so the
  theory is decidable.
\end{proof}

Using Example~\ref{acvnf-tp}, we deduce
\begin{corollary}
  Let $ACv^nF$ be the theory of algebraically closed fields with $n$
  valuation rings (as unary predicates).  Then the incomplete theory
  $ACv^nF$ is decidable.
\end{corollary}
\begin{proof}
  The only thing to check here is that we can bound the size of $P$
  from the number $n$ of given valuations $\Oo_1, \ldots, \Oo_n$.  On
  account of the tree structure, every valuation in $P$ is of the form
  $\Oo_i \cdot \Oo_j$ (or $K$), so there are certainly no more than
  $n^2 + 1$ elements in $P$.
\end{proof}

\section{Normal and relatively closed extensions} \label{cubes-galore}
\begin{definition}
  Fix a diagram
  \begin{equation*}
    \xymatrix{ F \ar[r] \ar[d] & K_1 \\ K_2 & }
  \end{equation*}
  in the category of fields.  A reduced amalgamation
  \begin{equation*}
    \xymatrix{ F \ar[r] \ar[d] & K_1 \ar[d] \\ K_2 \ar[r] & L }
  \end{equation*}
  is \emph{cozy} if the maps $K_i \to L$ are isomorphisms.  An
  amalgamation type is \emph{cozy} if a representative reduced
  amalgamation is cozy.
\end{definition}
\begin{remark}\label{cozy-remark}
  The following are equivalent:
  \begin{itemize}
  \item Every amalgamation type of $K_1$ and $K_2$ over $F$ is cozy.
  \item $K_1$ and $K_2$ are (algebraic) normal extensions of $F$,
    isomorphic to each other over $F$.
  \end{itemize}
\end{remark}
\begin{lemma}\label{normal-exts}
  Let $L/K$ be a normal (algebraic) extension, and $\Oo$ be a
  valuation ring on $K$.
  \begin{enumerate}
  \item $\Aut(L/K)$ acts transitively on the set of extensions of
    $\Oo$ to $L$.
  \item If $\Oo'$ is any extension of $\Oo$ to $L$, then the residue
    field extension is a normal extension.
  \item The residue field extension does not depend on $\Oo'$ in the
    following sense: if $\Oo'$ and $\Oo''$ are two extensions of $\Oo$
    to $L$, then $\res \Oo'$ and $\res \Oo''$ are isomorphic over
    $\res \Oo$.
  \end{enumerate}
\end{lemma}
\begin{proof}
  Let $\Oo_1$ and $\Oo_2$ be two (not necessarily distinct) extensions
  of $\Oo$ to $L$.  By Remark~\ref{cozy-remark} it suffices to show
  that $\Oo_1$ and $\Oo_2$ are in the same orbit of $\Aut(L/K)$ and
  that every amalgamation type of $\res \Oo_1$ and $\res \Oo_2$ over
  $\res \Oo$ is cozy.  Given any amalgamation type $\tau$, by
  Lemma~\ref{amalgamation-master} there is an amalgamation of valued
  fields
  \begin{equation*}
    \xymatrix{ (K,\Oo) \ar[r] \ar[d] & (L,\Oo_1) \ar[d] \\ (L,\Oo_2)
      \ar[r] & (L',\Oo')}
  \end{equation*}
  such that
  \begin{itemize}
  \item $L'$ is a reduced amalgamation of $L$ and $L$ over $K$
  \item $\res \Oo'$ is an amalgamation of $\res \Oo_1$ and $\res
    \Oo_2$ over $\res \Oo$, \emph{of type $\tau$}.
  \end{itemize}
  Then $L'$ is a cozy amalgamation of $L$ and $L$ over $K$, by
  normality of $L/K$, Remark~\ref{cozy-remark}, and the fact that $L'$
  is a reduced amalgamation.  If $\sigma$ is the induced isomorphism
  $L \stackrel{\sim}{\to} L' \stackrel{\sim}{\to} L$, then $\sigma \in
  \Aut(L/K)$ and $\sigma(\Oo_1) = \Oo_2$, proving transitivity.
  Moreover, the fact that $L'$ is a cozy amalgamation of $L$ and $L$
  over $K$ implies the same thing for the residue fields: $\res \Oo'$
  is a (reduced) cozy amalgamation of $\res \Oo_1$ and $\res \Oo_2$
  over $\res \Oo$.  Therefore $\tau$ is cozy, completing the proof.
\end{proof}
\begin{definition}
  If $L/K$ is a finite normal extension and $\Oo$ is a valuation ring
  on $K$, we let $n_{\Oo,L/K}$ denote the (finite) number of
  extensions of $\Oo$ to $L$.
\end{definition}
Note that $n_{\Oo,L/K}$ depends only on the isomorphism type of $L$
over $K$: if $L'/K$ is an isomorphic extension, then $n_{\Oo,L'/K} =
n_{\Oo,L/K}$.

Recall that if $L/K$ is a finite (algebraic) extension of valued
fields, then the residue field extension is also finite, of degree no
greater than $[L : K]$, because one can lift a basis of $\res L$ over
$\res K$ to a $K$-linearly independent set in $L$.
\begin{lemma}\label{central-diamond}
  Let
  \begin{equation*}
    \xymatrix{ L_1 \ar[r] & L_2 \\ K_1 \ar[u] \ar[r] & K_2 \ar[u]}
  \end{equation*}
  be a diagram of fields, in which $L_2/L_1$ and $K_2/K_1$ are finite
  normal extensions.  Suppose $K_1$ is relatively algebraically closed
  in $L_1$.  Let $\Oo_{L_1}$, $\Oo_{K_1}$, and $\Oo_{K_2}$ be
  valuation rings on $L_1$, $K_1$, and $K_2$, respectively.  Suppose
  $\Oo_{K_1}$ is the restriction of $\Oo_{L_1}$ and $\Oo_{K_2}$.
  Then, the set of valuation rings $\Oo_{L_2}$ on $L_2$ extending both
  $\Oo_{L_1}$ and $\Oo_{K_2}$ is non-empty, and has size exactly
  \begin{equation*}
    \frac{n_{\Oo_{L_1},L_2/L_1}}{n_{\Oo_{K_1},K_2/K_1}}
  \end{equation*}
\end{lemma}
\begin{proof}
  We claim that the restriction map $\rho : \Aut(L_2/L_1) \to
  \Aut(K_2/K_1)$ is surjective.\footnote{Here, $\Aut(-/-)$ denotes
    automorphisms of pure fields, not valued fields.}  Assume
  otherwise, and embed $L_2$ into a monster model $\Mm$ of ACF.  By
  elimination of imaginaries and the model-theoretic Galois
  correspondence, non-surjectivity implies there is $x \in \dcl(K_2)
  \cap \dcl(L_1) \setminus \dcl(K_1).$ Definable closure in ACF
  corresponds to perfect closure.  Thus, after replacing $x$ with
  $x^{p^k}$, we may assume $x \in K_2 \cap L_1$.  Then $K_2 \cap L_1
  \setminus K_1$ is non-empty, contradicting relative algebraic
  closure of $K_1$ in $L_1$.

  Thus $\rho : \Aut(L_2/L_1) \to \Aut(K_2/K_1)$ is surjective.  Let
  $V_L$ be the set of valuation rings on $L_2$ extending $\Oo_{L_1}$
  and $V_K$ be the set of valuation rings on $K_2$ extending
  $\Oo_{K_1}$.  Both these sets are finite.  By
  Lemma~\ref{normal-exts}, $\Aut(L_2/L_1)$ acts transitively on $V_L$
  and $\Aut(K_2/K_1)$ acts transitively on $V_K$.  The restriction map
  $V_L \to V_K$ is compatible with the action, in the sense that if
  $\Oo \in V_L$ and $\sigma \in \Aut(L_2/L_1)$, then
  \begin{equation*}
    (\sigma\cdot \Oo)) \cap K_2 = (\sigma | K_2) \cdot (\Oo \cap K_2).
  \end{equation*}
  In particular, if we view $V_K$ as an $\Aut(L_2/L_1)$-set via the
  homomorphism $\rho : \Aut(L_2/L_1) \to \Aut(K_2/K_1)$, then the
  restriction $V_L \to V_K$ is a homomorphism of $\Aut(L_2/L_1)$-sets.
  Then $V_K$ is a transitive $\Aut(L_2/L_1)$ by surjectivity of
  $\rho$.  Because both $V_L$ and $V_K$ are transitive
  $\Aut(L_2/L_1)$-sets, every fiber of the map $V_L \to V_K$ has the
  same cardinality.  This cardinality must be
  \begin{equation*}
    \frac{|V_L|}{|V_K|} =:
    \frac{n_{\Oo_{L_1},L_2/L_1}}{n_{\Oo_{K_1},K_2/K_1}}
  \end{equation*}
\end{proof}

\begin{lemma}\label{smooth-1}
  Let $L/K$ be a finite normal extension.  Let $\Oo_K \supseteq
  \Oo'_K$ be two valuation rings on $K$.  There is an integer
  $n_{\Oo_K,\Oo'_K,L/K}$ such that for any $\Oo_L$ on $L$ extending
  $\Oo_K$, the set $\mathcal{S}$ of $\Oo'_L \in \Val(L|\Oo_L)$
  extending $\Oo'_K$ has size exactly $n_{\Oo_K,\Oo'_K,L/K}$.

  Furthermore, for any $\Oo_L$, we have
  \begin{equation*}
    n_{\Oo_K,\Oo'_K,L/K} = n_{\Oo'_K \div \Oo_K,\res \Oo_L / \res
      \Oo_K}.
  \end{equation*}
\end{lemma}
\begin{proof}
  Let $k$ be the residue field of $\Oo_K$.  Given $\Oo_L$, let $\ell$
  be $\res \Oo_L$.  By Lemma~\ref{normal-exts}, the isomorphism type
  of $\ell$ over $k$ does not depend on $\Oo_L$.  Let
  $n_{\Oo_K,\Oo'_K,L/K}$ be $n_{\Oo'_K \div \Oo_K,\ell/k}$; this
  depends only on the isomorphism type of $\ell$ over $k$, hence is
  independent of $\Oo_L$.

  By Remark~\ref{meteor}.\ref{orchid} there is a diagram
  \begin{equation*}
    \xymatrix{ \Val(L|\Oo_L) \ar[r]^\sim \ar[d] &  \Val(\ell) \ar[d]
      \\ \Val(K|\Oo_K) \ar[r]_\sim & \Val(k)}
  \end{equation*}
  with horizontal maps the isomorphisms
  \begin{align*}
    \Oo & \mapsto \Oo \div \Oo_L \\
    \Oo & \mapsto \Oo \div \Oo_K
  \end{align*}
  respectively.

  The set $\mathcal{S}$ is the fiber of the left vertical map over
  $\Oo'_K$.  Via the horizontal isomorphisms, this is in bijection
  with the set of valuations on $\ell$ extending $\Oo'_K \div \Oo_K$.
  By definition, this set has size $n_{\Oo'_K \div \Oo_K, \ell / k}$,
  the value we chose for $n_{\Oo_K,\Oo'_K,L/K}$.
\end{proof}

\begin{lemma}\label{central-cube}
  Let
  \begin{equation*}
    \xymatrix{ L_1 \ar[r] & L_2 \\ K_1 \ar[u] \ar[r] & K_2 \ar[u]}
  \end{equation*}
  be a diagram of fields, in which $L_2/L_1$ and $K_2/K_1$ are finite
  normal extensions.  Let $\Oo_{L_2}$ be a valuation ring on $L_2$ and
  let $\Oo_{L_1}, \Oo_{K_2}, \Oo_{K_1}$ be the restrictions to $L_1,
  K_2$, and $K_1$, respectively.  Suppose that $\res \Oo_{K_1}$ is
  relatively algebraically closed in $\res \Oo_{L_1}$.  Let
  $\Oo'_{L_1}$, $\Oo'_{K_1}$, and $\Oo'_{K_2}$ be valuation rings on
  $L_1$, $K_1$, and $K_2$, respectively, such that
  \begin{itemize}
  \item $\Oo'_{L_1} \subseteq \Oo_{L_1}$, $\Oo'_{K_1} \subseteq
    \Oo_{K_1}$, and $\Oo'_{K_2} \subseteq \Oo_{K_2}$.
  \item $\Oo'_{K_1}$ is the restriction of both $\Oo'_{L_1}$ and
    $\Oo'_{K_2}$ to $K_1$.
  \end{itemize}
  Let $\mathcal{S}$ be the set of valuation rings $\Oo'_{L_2}$ on
  $L_2$ such that
  \begin{itemize}
  \item $\Oo'_{L_2} \subseteq \Oo_{L_2}$.
  \item $\Oo'_{L_2}$ extends both $\Oo'_{L_1}$ and $\Oo'_{K_2}$.
  \end{itemize}
  Then $\mathcal{S}$ is non-empty and has cardinality exactly
  \begin{equation*}
    \frac{n_{\Oo_{L_1},\Oo'_{L_1},L_2/L_1}}{n_{\Oo_{K_1},\Oo'_{K_1},K_2/K_1}}
  \end{equation*}
  where the $n$ are as in Lemma~\ref{smooth-1}.
\end{lemma}
\begin{proof}
  Let $\ell_i$ and $k_i$ denote the residue fields of $\Oo_{L_i}$ and
  $\Oo_{K_i}$, respectively.  By Remark~\ref{meteor}.\ref{orchid}
  there is a commutative diagram
  \begin{equation*}
    \xymatrix{
      \Val(L_1|\Oo_{L_1}) \ar[dd] \ar[dr] & & \Val(L_2|\Oo_{L_2}) \ar[ll] \ar[dr] \ar[dd] & \\
      & \Val(K_1|\Oo_{K_1}) \ar[dd] & & \Val(K_2|\Oo_{K_2}) \ar[ll] \ar[dd] \\
      \Val(\ell_1) \ar[dr] & & \Val(\ell_2) \ar[ll] \ar[dr] & \\
      & \Val(k_1) & & \Val(k_2) \ar[ll]
    }
  \end{equation*}
  with vertical maps bijections.  Under the bijection
  $\Val(L_2|\Oo_{L_2}) \to \Val(\ell_2)$, the set $\mathcal{S}$
  corresponds to the set of $\Oo$ on $\ell_2$ restricting to
  $\Oo'_{K_2} \div \Oo_{K_2}$ and $\Oo'_{L_1} \div \Oo_{L_1}$.  Now,
  by the commutative diagram, the fact that $\Oo'_{K_2}$ and
  $\Oo'_{L_1}$ both restrict to $\Oo'_{K_1}$ implies that $\Oo'_{K_2}
  \div \Oo_{K_2}$ and $\Oo'_{L_1} \div \Oo_{L_1}$ restrict to
  $\Oo'_{K_1} \div \Oo_{K_1}$.  By assumption, $k_1$ is relatively
  algebraically closed in $\ell_1$, so by Lemma~\ref{central-diamond},
  \begin{equation*}
    |\mathcal{S}| = \frac{n_{\Oo'_{L_1} \div
        \Oo_{L_1},\ell_2/\ell_1}}{n_{\Oo'_{K_1} \div
        \Oo_{K_1},k_2/k_1}}.
  \end{equation*}
  By Lemma~\ref{smooth-1},
  \begin{equation*}
    \frac{n_{\Oo'_{L_1} \div \Oo_{L_1},\ell_2/\ell_1}}{n_{\Oo'_{K_1}
        \div \Oo_{K_1},k_2/k_1}} =
    \frac{n_{\Oo_{L_1},\Oo'_{L_1},L_2/L_1}}{n_{\Oo_{K_1},\Oo'_{K_1},K_2/K_1}}.
  \end{equation*}
\end{proof}
\begin{definition}
  Let $P$ be a finite poset.
  \begin{enumerate}
  \item Write $x \rhd y$ if $x > y$ and there is no $z$ such that $x >
    z > y$.
  \item A \emph{choice system} on $P$ is a collection of sets
    $\mathcal{S}_x$ for $x \in P$ and relations $\mathcal{R}_{x,y}
    \subseteq \mathcal{S}_x \times \mathcal{S}_y$ for $x \rhd y$.
  \item Given a choice system on $P$ and a downwards closed subset $P'
    \subseteq P$, a \emph{partial choice} on $P'$ is a function $f$ on
    $P'$ such that
    \begin{equation*}
      \forall x \in P' : f(x) \in \mathcal{S}_x
    \end{equation*}
    \begin{equation*}
      \forall x \in P' \forall y \lhd x : f(x) \mathcal{R}_{x,y} f(y).
    \end{equation*}
    We write $\Gamma(P')$ for the collection of partial choices on
    $P'$.
  \item A choice system on $P$ is \emph{smooth at $x$} if there is a
    finite positive cardinal $n$ such that for any downward closed set
    $P' \subseteq P$ containing $x$ as a maximal element, every fiber
    of the restriction map
    \begin{equation*}
      \Gamma(P') \to \Gamma(P' \setminus \{x\})
    \end{equation*}
    has size $n$.
  \end{enumerate}
\end{definition}
\begin{remark}\label{the-point}
  Fix a choice system on a finite poset $P$, and let $P'$ be a
  downward closed subset of $P$.  If the choice system is smooth at
  every $x \in P \setminus P'$, then every fiber of the restriction
  map
  \begin{equation*}
    \Gamma(P) \to \Gamma(P')
  \end{equation*}
  has size $n$, for some finite positive $n$.
\end{remark}
\begin{theorem} \label{most-annoying}
  Fix a finite tree $P$.  Let $L/K$ be an extension of
  models of $T_P^0$.  Suppose that for every $p \in P$, the $p$th
  residue field extension $\res \Oo^L_p / \res \Oo^K_p$ is relatively
  algebraically closed.  Suppose we are given a diagram of pure fields
  \begin{equation*}
    \xymatrix{ L \ar[r] & L' \\ K \ar[u] \ar[r] & K' \ar[u]}
  \end{equation*}
  where $L'/L$ and $K'/K$ are finite normal extensions.  Let
  $\mathcal{S}_L$ and $\mathcal{S}_K$ be the set of extensions of the
  $T_P^0$-structures to $L'$ and $K'$, respectively.  Then
  \begin{enumerate}
  \item The sets $\mathcal{S}_L$ and $\mathcal{S}_K$ are finite.
  \item The restriction map $\mathcal{S}_L \to \mathcal{S}_K$ is
    surjective.
  \item Every fiber of this restriction map has the same size.
  \end{enumerate}
\end{theorem}
\begin{proof}
  Let $Q$ be the poset product of $P$ and the two-element total
  order $\{0,1\}$.  Note that all the relations $x \rhd y$ in $Q$ are
  of the following forms:
  \begin{itemize}
  \item $(x,1) \rhd (x,0)$.
  \item $(x,i) \rhd (y,i)$ where $i \in \{0,1\}$ and $y$ is the
    ``parent'' of $x$ in the tree $P$, i.e., $x \rhd y$.
  \end{itemize}
  We build a choice system on $Q$ as follows:
  \begin{itemize}
  \item $\mathcal{S}_{(x,0)}$ is the set of extensions (trivial if $x
    = \bot$) of $\Oo^K_x$ to $K'$.
  \item $\mathcal{S}_{(x,1)}$ is the set of extensions (trivial if $x
    = \bot$) of $\Oo^L_x$ to $L'$.
  \item If $\Oo^{L'}_x \in \mathcal{S}_{(x,1)}$ and $\Oo^{K'}_x \in
    \mathcal{S}_{(x,0)}$, then $\Oo^{L'}_x\mathcal{R}\Oo^{K'}_x$ holds
    iff $\Oo^{L'}_x$ extends $\Oo^{K'}_x$.
  \item If $y$ is the parent of $x$ in $P$, if $\Oo^{K'}_x \in
    \mathcal{S}_{(x,0)}$, and $\Oo^{K'}_y \in \mathcal{S}_{(y,0)}$,
    then $\Oo^{K'}_x \mathcal{R} \Oo^{K'}_y$ holds iff $\Oo^{K'}_x
    \subseteq \Oo^{K'}_y$.
  \item If $y$ is the parent of $x$ in $P$, if $\Oo^{L'}_x \in
    \mathcal{S}_{(x,0)}$, and $\Oo^{L'}_y \in \mathcal{S}_{(y,0)}$,
    then $\Oo^{L'}_x \mathcal{R} \Oo^{L'}_y$ holds iff $\Oo^{L'}_x
    \subseteq \Oo^{L'}_y$.
  \end{itemize}
  If $Q' = P \times \{0\}$, then a partial choice function on $Q'$ is
  an extension of the $T_P^0$-structure from $K$ to $K'$, and a
  partial choice function on $Q$ is an extension of the
  $T_P^0$-structure from $L$ to $L'$.  So it suffices to show that the
  choice system is smooth at every point $(x,i)$.

  The case where $x = \bot$ is easy, so assume $x > \bot$.  Let $y$ be
  the ``parent'' of $x$.  If $i = 0$, smoothness at $(x,0)$ follows by
  Lemma~\ref{smooth-1}.  Indeed, the number of valid choices for
  $\Oo^{K'}_x$ consistent with $\Oo^{K'}_y$ and $\Oo^K_x$ is exactly
  $n_{\Oo^K_y,\Oo^K_x,K'/K}$, which does not depend on the choices.

  Likewise, the case $i = 1$ follows by Lemma~\ref{central-cube}: the
  number of valid choices for $\Oo^{L'}_x$ consistent with
  $\Oo^{L'}_y$, $\Oo^{K'}_x$, and $\Oo^L_x$ is exactly
  \begin{equation*}
    \frac{n_{\Oo^L_y,\Oo^L_x,L'/L}}{n_{\Oo^K_y,\Oo^K_x,K'/K}}.
  \end{equation*}
  Again, this does not depend on the choices so far.
\end{proof}

\section{Probable truth} \label{probable-truth}
\begin{theorem} \label{prob-tr}
  There is a unique way to assign a probability
  $\Pp(\varphi(\vec{a})|K)$ to every model $K \models T_P^0$, tuple
  $\vec{a}$ from $K$, and formula $\varphi(\vec{a})$ in the language
  of $T_P$, satisfying the following properties:
  \begin{enumerate}
  \item \label{rat-01} $\Pp(\varphi(\vec{a})|K)$ is a rational number
    in $[0,1]$.
  \item \label{1-neg} $\Pp(\neg \varphi(\vec{a})|K) = 1 -
    \Pp(\varphi(\vec{a})|K)$.
  \item \label{additive} $\Pp(\varphi(\vec{a})|K) +
    \Pp(\psi(\vec{b})|K) = \Pp(\varphi(\vec{a}) \vee \psi(\vec{b}) |
    K) + \Pp(\varphi(\vec{a}) \wedge \psi(\vec{b}) | K)$.
  \item \label{exists} $\Pp(\varphi(\vec{a})|K) > 0$ if and only if $M
    \models \varphi(\vec{a})$ for at least one $T_P$-model $M \ge K$.
  \item \label{forall} $\Pp(\varphi(\vec{a})|K) = 1$ if and only if
    $M \models \varphi(\vec{a})$ for every $T_P$-model $M \ge K$.
  \item \label{weighting} If $L/K$ is a finite normal extension of
    pure fields, and $L_1, \ldots, L_n$ enumerate the
    $T_P^0$-structures on $L$ extending the given structure on $K$,
    then $\Pp(\varphi(\vec{a})|K)$ is the average of
    $\Pp(\varphi(\vec{a})|L_i)$.
  \item \label{isoinvar} If $f : K_1 \to K_2$ is an isomorphism of
    $T_P^0$-models, and $f(\vec{a}_1) = \vec{a}_2$, then
    $\Pp(\varphi(\vec{a}_1)|K_1) = \Pp(\varphi(\vec{a}_2)|K_2)$.
  \end{enumerate}
\end{theorem}
Conditions (\ref{rat-01}-\ref{forall}) say, among other things, that
$\Pp(-|K)$ defines a Keisler measure on the type space of embeddings
of $K$ into models of $T_P$.
\begin{proof}
  Let $\Pp'(-)$ be the partial function defined as follows:
  \begin{enumerate}
  \item $\Pp'(\varphi(\vec{a})|K) = 1$ if $M \models \varphi(\vec{a})$
    for every $T_P$-model $M \ge K$.
  \item $\Pp'(\varphi(\vec{a})|K) = 0$ if $M \models \neg
    \varphi(\vec{a})$ for every $T_P$-model $M \ge K$.
  \item $\Pp'(\varphi(\vec{a})|K)$ is undefined otherwise.
  \end{enumerate}
  Conditions \ref{exists} and \ref{forall} imply that $\Pp(-)$ must
  equal $\Pp'(-)$ when the latter is defined.

  Given $K$ and $\varphi(\vec{a})$, by Corollary~\ref{almost-qe} there
  is a finite normal extension $L/K$ such that for any $T_P$-model $M
  \ge K$, the truth of $M \models \varphi(\vec{a})$ is determined by
  the $T_P^0$-structure induced on $L$.  Let $L_1, \ldots, L_n$ be an
  enumeration of the distinct extensions of the $T_P^0$-structure from
  $K$ to $L$.  Thus $\Pp'(\varphi(\vec{a})|L_i)$ is defined for $i =
  1, \ldots, n$.  Then uniqueness of $\Pp(-)$ is clear: we must set
  \begin{equation*}
    \Pp(\varphi(\vec{a})|K) := \frac{\sum_{i = 1}^n
      \Pp'(\varphi(\vec{a})|L_i)}{n}.
  \end{equation*}
  It remains to show that this is well-defined and satisfies the
  required properties.

  Let $L'$ be another finite normal extension of $K$ which determines
  the truth of $\varphi(\vec{a})$.  We claim that $L$ and $L'$ yield
  the same value of $\Pp(\varphi(\vec{a})|K)$.  By relating $L$ to
  $LL'$ and $L'$ to $LL'$, we reduce to the case where $L' \ge L$.
  Applying Theorem~\ref{most-annoying} to the diagram
  \begin{equation*}
    \xymatrix{ K \ar[r] & L' \\ K \ar[u] \ar[r] & L \ar[u]}
  \end{equation*}
  there is an integer $m$ such that every $L_i$ has exactly $m$
  extensions $L'_{i,1}, \ldots, L'_{i,m}$ to a $T_P^0$-structure on
  $L'$.  Then $\{L'_{i,j}\}_{1 \le i \le n, ~ 1 \le j \le m}$ is an
  enumeration of the distinct $nm$-many $T_P^0$-structures on $L'$
  extending the given structure on $K$.  Moreover, every $T_P$-model
  extending $L'_{i,j}$ is a $T_P$-model extending $L_i$, so
  $\Pp'(\varphi(\vec{a})|L'_{i,j}) = \Pp'(\varphi(\vec{a})|L_i)$.
  Thus
  \begin{equation*}
    \frac{\sum_{i = 1}^n \sum_{j = 1}^m
      \Pp'(\varphi(\vec{a})|L'_{i,j})}{nm} = \frac{\sum_{i = 1}^n
      \Pp'(\varphi(\vec{a})|L_i)}{n}.
  \end{equation*}
  So the definition of $\Pp(\varphi(\vec{a})|K)$ using $L'$ agrees
  with that using $L$, and $\Pp(\varphi(\vec{a})|K)$ is well-defined.

  Condition (\ref{rat-01}) is clear, because we defined
  $\Pp(\varphi(\vec{a})|K)$ as an average of finitely many 0's and
  1's.  For Conditions (\ref{1-neg}-\ref{additive}), choose $L$ large
  enough that $\Pp'(\varphi(\vec{a})|L_i)$ and
  $\Pp'(\psi(\vec{b})|L_i)$ are well-defined for all $i$.  Then
  \begin{align*}
    \Pp'(\neg \varphi(\vec{a}) | L_i) &= 1 - \Pp'(\varphi(\vec{a})|L_i) \\
    \Pp'(\varphi(\vec{a}) \vee \psi(\vec{b}) | L_i) &=
    \max(\Pp'(\varphi(\vec{a})|L_i),\Pp'(\psi(\vec{b})|L_i)) \\
    \Pp'(\varphi(\vec{a}) \wedge \psi(\vec{b}) | L_i) &=
    \min(\Pp'(\varphi(\vec{a})|L_i),\Pp'(\psi(\vec{b})|L_i))
  \end{align*}
  for all $i$---in particular the left hand sides are well-defined.
  The desired equations then follow by averaging
  \begin{align*}
    \Pp'(\neg \varphi(\vec{a}) | L_i) &= 1 - \Pp'(\varphi(\vec{a})|L_i) \\
    \Pp'(\varphi(\vec{a}) | L_i) + \Pp'(\psi(\vec{b}) | L_i) &=
    \Pp'(\varphi(\vec{a}) \vee \psi(\vec{b}) | L_i) +
    \Pp'(\varphi(\vec{a}) \wedge \psi(\vec{b}) | L_i)
  \end{align*}
  over $i = 1, \ldots, n$.

  For (\ref{exists}), note that $\Pp(\varphi(\vec{a})|K) > 0$ if and
  only if $\Pp'(\varphi(\vec{a})|L_i) = 1$ for at least one $i$.  If
  this holds, then extending $L_i$ to a $T_P$-model $M$, we obtain a
  $T_P$-model $M$ extending $K$ in which $\varphi(\vec{a})$ holds.
  Conversely, if $\Pp'(\varphi(\vec{a})|L_i) = 0$ for all $i$, and $M$
  is any $T_P$-model extending $K$, then $M$ extends some $L_i$, and
  so $M \models \neg \varphi(\vec{a})$.

  Thus (\ref{exists}) holds.  Condition (\ref{forall}) follows from
  (\ref{exists}) and (\ref{1-neg}).

  Next consider the situation of (\ref{weighting}).  We can find a
  normal extension $L'$ of $K$ such that $L' \ge L$ and
  $\Pp'(\varphi(\vec{a})|L')$ is defined for any extension of the
  $T_P^0$-structure to $L'$.  As before, by an application of
  Theorem~\ref{most-annoying} we know that there is an integer $m$
  such that every $L_i$ has exactly $m$ extensions $L'_{i,1}, \ldots,
  L'_{i,m}$ to a $T_P^0$-structure on $L'$.  Then
  \begin{equation*}
    \frac{\sum_{i = 1}^n \Pp(\varphi(\vec{a})|L_i)}{n} = \frac{\sum_{i
        = 1}^n \frac{\sum_{j = 1}^m
        \Pp'(\varphi(\vec{a})|L'_{i,j})}{m}}{n} = \frac{\sum_{i = 1}^n
      \sum_{j = 1}^m \Pp'(\varphi(\vec{a})|L'_{i,j})}{nm} =
    \Pp(\varphi(\vec{a})|K).
  \end{equation*}
  Thus (\ref{weighting}) holds.  Finally, (\ref{isoinvar}) is clear
  from the definition.
\end{proof}

\begin{proposition}\label{key-technical}
  Let $L/K$ be an extension of models of $T_P^0$ with the following
  property: for every $p \in P$, the residue field extension $\res
  \Oo^L_p / \res \Oo^K_p$ is relatively algebraically closed.  Then
  for any formula $\varphi(\vec{a})$ with parameters $\vec{a}$ from
  $K$, we have
  \begin{equation*}
    \Pp(\varphi(\vec{a})|L) = \Pp(\varphi(\vec{a})|K).
  \end{equation*}
\end{proposition}
\begin{proof}
  As in the proof of Theorem~\ref{prob-tr}, let $\Pp'(\varphi|K)$ be
  0, 1, or undefined, depending on whether $\varphi$ holds in none,
  all, or some of the models of $T_P$ extending $K$.  Using
  Corollary~\ref{almost-qe}, choose a finite normal extension $K'$ of
  $K$ such that $\Pp'(\varphi(\vec{a})|K')$ is defined for every
  $T_P^0$-structure on $K'$ extending the given structure on $K$.  Let
  $L' = LK'$.  Let $K'_1, \ldots, K'_n$ enumerate the
  $T_P^0$-structures on $K'$ extending $K$.  By
  Theorem~\ref{most-annoying}, there is an integer $m$ such that for
  every $K'_i$, there are exactly $m$-many $T_P^0$-structures
  $L'_{i,1}, \ldots, L'_{i,m}$ on $L'$ extending $K'_i$ and $L$.  Note
  that $\{L'_{i,j}\}_{1 \le i \le n,~ 1 \le j \le m}$ is an exhaustive
  listing of the distinct $T_P^0$-structures on $L'$ extending $L$.

  For any $i, j$, note that $\Pp'(\varphi(\vec{a})|L'_{i,j})$ is
  defined and equals $\Pp'(\varphi(\vec{a})|K'_i)$.  Indeed, if $M$ is
  a model of $T_P$ extending $L'_{i,j}$, then $M$ is a model of $T_P$
  extending $K'_i$, so whether $M \models \varphi(\vec{a})$ must agree
  with $\Pp'(\varphi(\vec{a})|K'_i)$.  So
  \begin{equation*}
    \Pp'(\varphi(\vec{a})|L'_{i,j}) = \Pp'(\varphi(\vec{a})|K'_i).
  \end{equation*}
  Averaging over all $i$ and $j$ immediately implies
  \begin{equation*}
    \Pp(\varphi(\vec{a})|L) = \Pp(\varphi(\vec{a})|K).
  \end{equation*}
\end{proof}

\section{Indiscernible sequences and relative closure} \label{sec:coheirs}
\begin{lemma}\label{the-lemma}
  Let $\Mm$ be a valued field that is a monster model of either ACVF
  or ACF with the trivial valuation.  Let $a$ be a tuple and let
  \begin{align*}
    \ldots, b^-_{-1}, & b^-_0, b^-_1, \ldots, \\
    \ldots, b_{-1}, & b_0, b_1, \ldots, \\
    \ldots, b^+_{-1}, & b^+_0, b^+_1, \ldots
  \end{align*}
  be an $a$-indiscernible sequence of tuples, of length $3 \times
  \Zz$.  For $S \subseteq \Zz$ let $b_S$ denote $\{b_i : i \in S\}$
  and similarly for $b^+_S$ and $b^-_S$.  Let $K_i$ and $L$ be the
  algebraically closed subfields of $\Mm$ generated by $b^-_\Zz b_i
  b^+_\Zz$ and $b^-_\Zz b_\Zz b^+_\Zz$, respectively.  Abusing
  notation slightly, let $K_i(a)$ and $L(a)$ denote the \emph{perfect}
  subfields of $\Mm$ generated by $aK_i$ and $aL$, respectively.  Then
  the residue field of $K_i(a)$ is relatively algebraically closed in
  the residue field of $L(a)$.
\end{lemma}
\begin{proof}
  Without loss of generality, $i = 0$.  Let $k$ be the residue field
  of $K_0(a)$ and $\ell$ be the residue field of $L(a)$.  Take $\alpha
  \in \ell \cap k^{alg} \setminus k$.  Let $S$ be the set of roots of
  the minimal polynomial of $\alpha$ over $k$.  This is a
  $K_0(a)$-definable finite set, so it is $F(a)$-definable where $F =
  \acl(b^-_S b_0 b^+_S)$, for some finite $S \subseteq \Zz$.  Because
  $\alpha \in \ell$, we can write $\alpha$ as
  \begin{equation*}
    \alpha = \res \frac{P(a,c)}{Q(a,c)}
  \end{equation*}
  where $P, Q$ are polynomials with integral coefficients and $c$ is a
  tuple from $L = \acl(b^-_\Zz b_\Zz b^+_\Zz)$.  Increasing $S$, we
  may assume $c \in \acl(b^-_S b_S b^+_S)$ and $0 \in S$.  Let
  \begin{equation*}
    i_1 < \cdots < i_n < 0 < j_1 < \cdots < j_m
  \end{equation*}
  be the elements of $S$ in order.  Note
  that the two sequences
  \begin{align*}
    b^-_{j_m + 1}, b^-_{j_m + 2} \ldots, & \ldots, b_{-3}, b_{-2}, b_{-1} \\
    b_1, b_2, b_3, \ldots, & \ldots, b^+_{i_1 - 2}, b^+_{i_1 - 1}
  \end{align*}
  are mutually indiscernible over $ab^-_Sb_0b^+_S$, hence over
  $F(a)^{alg}$.\footnote{In general, indiscernibility over $A$
    is the same thing as indiscernibility over $\acl(A)$.}  Choose
  $i'_1 < \cdots < i'_n$ greater than $j_m$ and $j'_1 < \cdots < j'_m$
  less than $i_1$.  Then
  \begin{align*}
    b_{i_1} \cdots b_{i_n} b_{j_1} \cdots b_{j_m} \equiv_{F(a)^{alg}} 
    b^-_{i'_1}\cdots b^-_{i'_n} b^+_{j'_1} \cdots b^+_{j'_m}
  \end{align*}
  by the mutual indiscernibility.  Let $\sigma \in
  \Aut(\Mm/F(a)^{alg})$ be an automorphism moving the left hand side
  to the right hand side.  Then
  \begin{equation*}
    \sigma(b^-_S b_S b^+_S) = b^-_S b^-_{i'_1} \cdots b^-_{i'_n} b_0
    b^+_{j'_1} \cdots b^+_{j'_m} b^+_S
  \end{equation*}
  and so
  \begin{equation*}
    \sigma(c) \in \acl(b^-_S b^-_{i'_1} \cdots b^-_{i'_n} b_0
    b^+_{j'_1} \cdots b^+_{j'_m} b^+_S) \subseteq \acl(b^-_\Zz b_0 b^+_\Zz)
  \end{equation*}
  so $\sigma(c)$ is a tuple from $K_0$.  Thus $\sigma(\alpha)$ is a
  residue from $K_0(a)$.  Now $\sigma$ fixes $S$ setwise, so $S$
  intersects $k$, a contradiction.
\end{proof}

\section{Finite burden} \label{finite-burden}

\begin{theorem}
  Let $N$ be the number of ``leaves'' in $P$, i.e., maximal elements.
  Then $T_P$ has burden no more than $2N$.\footnote{The argument could
    probably be improved to get $N$ rather $2N$.}
\end{theorem}
\begin{proof}
Otherwise, take a mutually indiscernible inp-pattern with $3 \times
\Zz$ columns and $2N+1$ rows.  Let the $i$th row be
\begin{align*}
  \ldots, \varphi_i(x;b^-_{i,-1}), & \varphi_i(x;b^-_{i,0}), \varphi_i(x;b^-_{i,1}), \ldots, \\
  \ldots, \varphi_i(x;b_{i,-1}), & \varphi_i(x;b_{i,0}), \varphi_i(x;b_{i,1}), \ldots, \\
  \ldots, \varphi_i(x;b^+_{i,-1}), & \varphi_i(x;b^+_{i,0}), \varphi_i(x;b^+_{i,1}), \ldots  
\end{align*}
Let $B^\pm_i$ denote the set $\{\ldots, b^\pm_{i,-1}, b^\pm_{i,0},
b^\pm_{i,1}, \ldots\}$.
\begin{claim}
  There is a mutually indiscernible array $c_{i,j}$ of infinite
  tuples, such that $c_{i,j}$ is an enumeration of $\acl(b_{i,j}B^+_i
  B^-_i)$.
\end{claim}
\begin{proof}
  Let $Q = \bigcup_i (B^+_i \cup B^-_i)$.  Note that the $b_{i,j}$
  form a mutually indiscernible array over $Q$.  Take $\hat{c}_{i,0}$
  to be an enumeration of $\acl(b_{i,0}B^+_i B^-_i)$ and choose
  $\hat{c}_{i,j}$ so that $\hat{c}_{i,0}b_{i,0}$ has the same type as
  $\hat{c}_{i,j}b_{i,j}$ over $Q$.  Let $\{e_{i,j}d_{i,j}\}$ be a
  mutually indiscernible array over $Q$ extracted from
  $\{\hat{c}_{i,j}b_{i,j}\}$.  As $b_{i,j}$ was already mutually
  indiscernible over $Q$, the array $\{d_{i,j}\}$ has the same type as
  $\{b_{i,j}\}$ over $Q$.  Choose $\sigma \in \Aut(\Mm/Q)$ such that
  $\sigma(d_{i,j}) = b_{i,j}$, and set $\tilde{c}_{i,j} =
  \sigma(e_{i,j})$.  Then the $\{\tilde{c}_{i,j}\}$ are mutually
  $Q$-indiscernible because $\{e_{i,j}\}$ are.

  Because $\tp(\hat{c}_{i,j}b_{i,j}/Q)$ was
  $\tp(\hat{c}_{i,0}b_{i,0}/Q)$ for all $j$, the same holds for the
  extracted array: $\tp(e_{i,j}d_{i,j}/Q) =
  \tp(\hat{c}_{i,0}b_{i,0}/Q)$ for all $i, j$.  Therefore
  \[\tilde{c}_{i,j}b_{i,j} \equiv_Q e_{i,j}d_{i,j} \equiv_Q
  \hat{c}_{i,0}b_{i,0}.\] By choice of $\hat{c}_{i,0}$, it follows
  that $\tilde{c}_{i,j}$ is an enumeration of $\acl(b_{i,j} B^+_i
  B^-_i)$.
\end{proof}  

Fix some element $a$ such that $\varphi_i(a;b_{i,0})$ holds
for all $i$.

For each $p \in P$, consider the reduct of $\Mm$ to $(K, \mathcal{O}_q
: q \le p)$.  This reduct is a model of the theory of algebraically
closed fields with $(|[\bot,p]|-1)$-many \emph{comparable} valuations.
This theory is an expansion of ACVF by externally definable sets (in
the value group), so it has dp-rank 1.

Recall that $N$ is the number of minimal elements in $P$.  By Lemma
4.1 in \cite{dp-add}, we can drop no more than $2N$ rows and arrange
that
\begin{itemize}
\item Each row
  \begin{equation*}
    \ldots, b^-_{i,0}, \ldots, b_{i,0}, \ldots, b^+_{i,0}, \ldots
  \end{equation*}
  is $a$-indiscernible in every reduct $(\Mm,\mathcal{O}_p)$.
\item Each row
  \begin{equation*}
    \ldots, c_{i,-1}, c_{i,0}, c_{i,1}, \ldots
  \end{equation*}
  is $a$-indiscernible in every reduct $(\Mm,\mathcal{O}_p)$.
\end{itemize}
Since we started with $2N+1$ rows, at least one row remains.  Focus on
this one row, and drop the subscript $i$'s.  We now have the following
configuration:
\begin{enumerate}
\item The sequence
  \begin{align*}
    \ldots, b^-_{-1}, & b^-_0, b^-_1, \ldots, \\
    \ldots, b_{-1}, & b_0, b_1, \ldots, \\
    \ldots, b^+_{-1}, & b^+_0, b^+_1, \ldots
  \end{align*}
  is $a$-indiscernible in every reduct $(\Mm,\mathcal{O}_p)$.
\item The sequence
  \begin{equation*}
    \ldots, c_{-1}, c_0, c_1, \ldots
  \end{equation*}
  is $a$-indiscernible in every reduct $(\Mm,\mathcal{O}_p)$.
\item Each $c_i$ is an enumeration of $\acl(b^-_\Zz b_i b^+_\Zz)$.
\item The set of formulas
  \begin{equation*}
    \ldots, \varphi(x;b_{-1}), \varphi(x;b_0), \varphi(x;b_1), \ldots
  \end{equation*}
  is $k$-inconsistent.
\item $\varphi(a;b_0)$ holds.
\end{enumerate}
As in Lemma~\ref{the-lemma}, let $K_i$ and $L$ be the algebraically
closed subfields of $\Mm$ generated by $b^-_\Zz b_i b^+_\Zz$, and
$b^-_\Zz b_\Zz b^+_\Zz$, and let $K_i(a)$ and $L(a)$ denote the
perfect closures when $a$ is thrown in.  Note that $c_i$ is an
enumeration of $K_i$.  For any $p \in P$, the $p$th residue field of
$K_i(a)$ is relatively algebraically closed in $L(a)$ by
Lemma~\ref{the-lemma}.  Then by Proposition~\ref{key-technical},
\begin{equation*}
  \Pp(\varphi(a;b_i)|K_i(a)) = \Pp(\varphi(a;b_i)|L(a)).
\end{equation*}
Now for any $i, j$, there is an isomorphism of multi-valued fields
from $K_i(a)$ to $K_j(a)$ sending $a$ to itself and $c_i$ to $c_j$.
This holds because $c_i \equiv_a c_j$ in each reduct $(\Mm,\Oo_p)$.
It follows that
\begin{equation*}
  \Pp(\varphi(a;b_i)|L(a)) = \Pp(\varphi(a;b_0)|K_0(a)) > 0
\end{equation*}
for all $i$, where the inequality holds because $\Mm \models
\varphi(a;b_0)$.

Now take $N$ so large that $N \cdot \Pp(\varphi(a;b_0)|K_0(a)) > k$.
Then
\begin{equation*}
  \Pp(\varphi(a;b_1)|L(a)) + \Pp(\varphi(a;b_2)|L(a)) + \cdots +
  \Pp(\varphi(a;b_N)|L(a)) > k
\end{equation*}
so by a simple probabilistic argument it follows that there is some small
$M \models T_P$ extending $L(a)$ such that
\begin{equation*}
  |\{i \in \{1,\ldots,N\} : M \models \varphi(a;b_i)\}| \ge k + 1.
\end{equation*}
Since $L$ is algebraically closed, we can find some embedding of $M$
into $\Mm$ over $L$.  Let $a'$ be the image of this embedding.  By
model completeness,
\begin{equation*}
  M \models \varphi(a;b_i) \iff \Mm \models \varphi(a';b_i).
\end{equation*}
So $\varphi'(a';b_i)$ holds for at least $k+1$ values of $i$,
contradicting $k$-inconsistency.
\end{proof}

By Example~\ref{acvnf-tp},
\begin{corollary}
  If $(K,\Oo_1,\ldots,\Oo_n)$ is an algebraically closed field
  expanded with $n$ valuation rings, the resulting structure has
  finite burden.
\end{corollary}

\section{NIP, or lack thereof} \label{sec:nip}
\begin{lemma}
  Let $(K,\Oo_1,\Oo_2)$ be an algebraically closed field with two
  independent non-trivial valuation rings.  Then $(K,\Oo_1,\Oo_2)$ is
  not NIP, i.e., $(K,\Oo_1,\Oo_2)$ has the independence property.
\end{lemma}
\begin{proof}
  Let $p$ be a prime distinct from the characteristics of $K$, $\res
  \Oo_1$, and $\res \Oo_2$.  Let $\omega \in K$ be a primitive $p$th
  root of unity.  Abusing notation, we also let $\omega$ denote its
  residues in $\res \Oo_1$ and $\res \Oo_2$.  Let $\mm_i$ denote the
  maximal ideal of $\Oo_i$.  For $k \in \Zz/p\Zz$, let $U_k$ and $V_k$
  denote $\omega^k + \mm_1$ and $\omega^k + \mm_2$.  Note that the
  $U_k$ are pairwise disjoint and their union is the set of $x$ such
  that $x^p \in U_0$.  Similarly, the $V_k$ are pairwise disjoint and
  their union is the set of $x$ such that $x^p \in V_0$.

  Let $W$ be the definable set $\{x^p : x \in U_0 \cap V_0\}$.  We
  claim that the relation
  \begin{equation*}
    \varphi(x;y) \iff x + y \in W
  \end{equation*}
  has the independence property.  Let $\epsilon_1, \ldots, \epsilon_n$
  be $n$ distinct elements in $\mm_1 \cap \mm_2$.  Consider the affine
  variety $C$ in $n+1$ variables $(x_1,\ldots,x_n,y)$ cut out by the
  equations
  \begin{equation*}
    x_i^p = y + \epsilon_i
  \end{equation*}
\begin{claim}\label{irredc}
    $C$ is irreducible.
  \end{claim}
  \begin{proof}
    It suffices to show that the ring
    \begin{equation*}
      K[X_1,\ldots,X_n,Y]/(X_1^p - Y - \epsilon_1, X_2^p - Y -
      \epsilon_2, \ldots, X_n^p - Y - \epsilon_n)
    \end{equation*}
    is an integral domain.  This follows from the more general
    property: if $R$ is a unique factorization domain, if $F =
    \Frac(R)$, if $p \in \Nn$ is a prime distinct from
    $\characteristic(F)$, if $R$ contains a primitive $p$th root of
    unity, and if $q_1,\ldots,q_n$ are elements of $R$ generating
    distinct prime ideals, then $S : = R[X_1,\ldots,X_n]/(X_1^p - q_1,
    \ldots, X_n^p - q_n)$ is an integral domain.  First note that $S$
    is a free $R$-module with basis the monomials $X_1^{s_1} \cdots
    X_n^{s_n}$ with $0 \le s_n < p$.  Therefore $S$ injects into
    \begin{equation*}
      S' := S \otimes_R F = F[X_1,\ldots,X_n]/(X_1^p - q_1, \ldots,
      X_n^p - q_n).
    \end{equation*}
    Let $L$ be the Galois extension of $F$ obtained by adding $p$th
    roots to $q_1, \ldots, q_n$; this is Galois because $R$ has the
    primitive $p$th roots of unity.  Then $S'$ and $L$ are finite
    $F$-algebras, and there is a surjection $S' \twoheadrightarrow L$.
    It suffices to show that $\dim_F S' = [L : F]$.  There is an
    injection $\Gal(L/F) \to (\Zz/p\Zz)^n$ determined by the faithful
    action of $\Gal(L/F)$ on the $p$th roots of the $q_i$.  If this
    injection fails to be onto, we can find a nonzero vector
    $(s_1,\ldots,s_n) \in (\Zz/p\Zz)^n$ complementary to the image.
    Then $\Gal(L/F)$ fixes $t = \prod_{i = 1}^n q_i^{s_i/p}$, so $t
    \in F$.  But then
    \begin{equation*}
      t^p = \prod_{i = 1}^n q_i^{s_i}
    \end{equation*}
    is a $p$th power in $F$, contradicting unique factorization in
    $R$, as the $s_i$ are not all congruent to 0 modulo $p$.

    Unwinding, it follows that the image of $\Gal(L/F) \to
    (\Zz/p\Zz)^n$ is all of $(\Zz/p\Zz)^n$, so $\Gal(L/F)$ has size at
    least $p^n$, so $[L : F] \ge p^n = \dim_F S'$, so $S'
    \twoheadrightarrow L$ is an isomorphism, so $S'$ is a field, so
    $S$ is an integral domain.
  \end{proof}
  For any function $\eta : [n] \to [p]$, let $U_\eta$ be the set of
  $(\vec{x},y) \in C$ such that $x_i \in U_{\eta(i)}$ for every $i$.
  \begin{claim}
    For any $\eta$, the set $U_\eta$ is non-empty.
  \end{claim}
  \begin{proof}
    Take arbitrary $y \in 1 + \mm_1 \cap \mm_2$.  It suffices to prove
    that for any $i, k$, there is an $x_i \in U_k$ such that $x_i^p =
    y + \epsilon_i$.  Because $K$ is algebraically closed, $y$ has at
    least one $p$th root $z$.  One checks that $z \in U_{k'}$ for some
    $k'$.  Multiplying by $\omega^{k - k'}$ yields a $p$th root in
    $U_k$.
  \end{proof}
  Similarly, define $V_\eta$ to be the set of $(\vec{x},y) \in C$ such
  that $x_i \in V_{\eta(i)}$ for every $i$.  Then $V_\eta$ is likewise
  non-empty.  Note that $U_\eta$ and $V_\eta$ are open subsets of $C$
  with respect to the topologies induced by $\Oo_1$ and $\Oo_2$,
  \emph{respectively}.  By Fact~\ref{strong-approx} and
  Claim~\ref{irredc} above, it follows that $U_\eta \cap V_{\eta'} \ne
  \emptyset$ for any $\eta, \eta'$.  Now given $S \subseteq
  \{1,\ldots,n\}$, choose $\eta, \eta'$ such that
  \begin{equation*}
    S = \{i : \eta(i) = \eta'(i)\}
  \end{equation*}
  and choose $(\vec{x},y) \in U_\eta \cap V_{\eta'}$.
  \begin{claim}
    For any $i$,
    \begin{equation*}
      y + \epsilon_i \in W \iff \eta(i) = \eta'(i) \iff i \in S.
    \end{equation*}
  \end{claim}
  \begin{proof}
    First suppose $\eta(i) = \eta'(i) = k$.  Then $x_i \in U_k \cap
    V_k$, so $\omega^{-k}x_i \in U_0 \cap V_0$.  Thus $y + \epsilon_i$
    is the $p$th power of an element of $U_0 \cap V_0$, namely
    $\omega^{-k}x_i$.  So $y + \epsilon_i \in W$ by definition of $W$.

    Conversely, suppose $y + \epsilon_i \in W$.  Then there is some $z
    \in U_0 \cap V_0$ such that $z^p = y + \epsilon_i = (x_i)^p$.  So
    $x_i = \omega^k z$ for some $k \in \Zz/p\Zz$.  Then $x_i \in U_k
    \cap V_k$, so $\eta(i) = k = \eta'(i)$.
  \end{proof}
  This last claim immediately implies that the relation
  \begin{equation*}
    \varphi(x;y) \iff x + y \in W
  \end{equation*}
  has the independence property.
\end{proof}

\begin{theorem}
  A model $K \models T_P$ is NIP if and only if $P$ is totally
  ordered.
\end{theorem}
\begin{proof}
  First suppose $P$ is not totally ordered.  Take two incomparable
  elements $p_1, p_2$ and let $p_0 = p_1 \wedge p_2$.  Then
  $\Oo^K_{p_0}$ is the join of $\Oo^K_{p_1}$ and $\Oo^K_{p_2}$.  It
  follows that
  \begin{equation*}
    \Oo^K_{p_1} \div \Oo^K_{p_0}, ~ \Oo^K_{p_2} \div \Oo^K_{p_0}
  \end{equation*}
  are two independent non-trivial valuations on $\res \Oo^K_{p_0}$.
  These valuation rings and the field $\Oo^K_{p_0}$ are interpretable,
  so $K$ interprets an algebraically closed field with two independent
  valuations, and therefore fails NIP by the Lemma.

  Conversely, suppose $P$ is totally ordered.  Let $\top$ be the
  greatest element of $P$.  Then every $\Oo_p$ is a coarsening of
  $\Oo_\top$.  Let $\Gamma$ be the value group of $\Oo_\top$.  The
  two-sorted structure $(K,\Oo_\top,\Gamma)$ is bi-interpretable with
  the C-minimal theory ACVF, hence NIP.  Every convex subgroup of
  $\Gamma$ is externally definable.  Therefore, in the Shelah
  expansion of $(K,\Oo_0,\Gamma)$, every convex subgroup of $\Gamma$
  is definable, and every coarsening of $\Oo_\top$ is definable.
  Consequently, the original structure $(K,\Oo_p : p \in P)$ is
  interpretable in the (NIP) Shelah expansion of $(K,\Oo_\top,\Gamma)$.
\end{proof}
\begin{corollary}
  A structure $(K,\Oo_1,\ldots,\Oo_n)$ with $K = K^{alg}$ is NIP iff
  the $\Oo_i$ are pairwise comparable.
\end{corollary}

\section{Open questions} \label{sec:horizons}
From here, there are several evident directions for potential
generalization.
\subsection{Improving the bound on burden}
If $P$ is a tree with $n$ leaves, we have shown that $T_P$ has burden at most $2n$.  This is probably suboptimal; the correct value should be $n$.
\subsection{Multi-valued fields with residue structure}
If the $T_i$ in Lemma~\ref{mildly-annoying} have finite burden, must the model companion then have finite burden?  If so, this would give a more direct proof that $T_P$ has finite burden.
\subsection{Forking and dividing}
Can we characterize forking in the theory $T_P$?  Does forking equal
dividing?  In the case where $P = \{\bot, 1, \ldots, n\}$, i.e., the
case of $n$ independent non-trivial valuations, forking was
characterized in \cite{myself} \S 11.6.  Specifically, $A \forkindep_B
C$ holds in the structure $(K,\Oo_1,\ldots,\Oo_n)$ if and only if $A
\forkindep_B C$ holds in each ACVF reduct $(K,\Oo_i)$.  Moreover,
forking equals dividing.  It would be natural to generalize these
results to the non-independent setting.

\subsection{Real closed and $p$-adically closed fields}
Chapter 11 of \cite{myself} also considered the setting of
$(K,\Oo_1,\ldots,\Oo_n)$, where $K$ is real closed or $p$-adically
closed and the $\Oo_i$ are independent non-trivial valuation rings,
independent from the canonical topology on $K$.  Under these
assumptions, the structure has finite burden.  It seems that one
should be able to drop these independence assumptions.  For example,
the theory of real closed fields $(K,+,\cdot,\Oo_1,\ldots,\Oo_n)$ with
$n$ valuation rings ought to have finite burden and be decidable.

The appropriate analogue of $T_P$ should be the following.  Let $P$ be
a non-trivial finite tree and $\rho$ be a distinguished leaf (maximal
element).  Define $T^R_{(P,\rho)}$ recursively as follows.  Let $P_1,
\ldots, P_n$ be the branches of $P$; without loss of generality $P_1$
is the branch containing $\rho$.
\begin{itemize}
\item If $P_1 = \{\rho\}$, then a model of $T^R_{(P,\rho)}$ should consist of
  \begin{enumerate}
  \item A real closed field $K$.
  \item Non-trivial valuation rings $\Oo_2, \ldots, \Oo_n$ on $K$,
    independent from each other and from the order topology on $K$.
    (This implies that each $\res \Oo_i$ is algebraically closed.)
  \item A $T_{P_i}$ structure on each $\res \Oo_i$.
  \end{enumerate}
\item If $P_1$ is non-trivial, then a model of $T^R_{(P,\rho)}$ should
  consist of
  \begin{enumerate}
  \item A real closed field $K$.
  \item Non-trivial independent valuation rings $\Oo_1, \ldots,
    \Oo_n$, where $(K,\Oo_1) \models RCVF$, i.e., $\Oo_1$ is a convex
    subgroup.  (This ensures that $\res \Oo_1$ is real closed and
    $\res \Oo_i$ is algebraically closed for $i > 1$.)
  \item A $T^R_{(P_1,\rho)}$-structure on $\res \Oo_1$.
  \item A $T_{P_i}$-structure on $\res \Oo_i$ for $i > 1$.
  \end{enumerate}
\end{itemize}
Something similar should work for $p$-adically closed fields.

\subsection{Bounded PRC and P$p$C fields}
Let $T_{n,m}$ be the theory of existentially closed fields with $n$
valuations and $m$ orderings.  In Chapter 11 of \cite{myself}, the
theory $T_{n,m}$ was shown to have finite burden.  The case $T_{n,1}$
is the aforementioned real closed field with $n$ independent
valuations.

The case $T_{0,m}$ of $m$ orderings and no valuations is a special
case of a theorem of Montenegro.  Recall that a field $K$ is
\emph{bounded} if it has finitely many Galois extensions of degree
$d$, for every $d$.  In her dissertation \cite{samaria-thesis},
Montenegro proved that bounded pseudo real closed (PRC) fields have
finite burden.  The models of $T_{0,m}$ turn out to be a subset of the
bounded PRC fields.

A model of $T_{n,m}$ is probably equivalent to a model of $T_{0,m}$
with $n$ independent valuation rings, independent from the order
topologies.  More generally, there is work in progress by Montenegro
and Rideau-Kikuchi which should show that if $K$ is a bounded
pseudo-real closed field, and $\Oo_1, \ldots, \Oo_n$ are $n$
independent valuations on $K$, independent from all the orderings,
then $(K,\Oo_1,\ldots,\Oo_n)$ has finite burden.

It would be natural to ask whether the independence assumption can be
dropped: given a bounded PRC field $K$ and finitely many arbitrary
valuation rings $\Oo_1, \ldots, \Oo_n$, does the resulting structure
$(K,\Oo_1,\ldots,\Oo_n)$ have finite burden?

More generally, one can replace ``PRC'' with ``pseudo $p$-adically
closed'', or a mixture of the two, and the above discussion goes
through (including the citations).

\subsection{Dp-minimal fields}
After [sic] \cite{samaria-thesis} and \cite{myself}\S 11 were
completed, dp-minimal fields were completely classified (\cite{myself}
\S 9).  In the preceding discussions, can we replace RCF and $p$CF
with other dp-minimal theories of fields?  For example,
\begin{conjecture}\label{the-conj}
  If $K$ is a dp-minimal pure field and $\Oo_1, \ldots, \Oo_n$ are
  valuation rings on $K$, then $(K,\Oo_1,\ldots,\Oo_n)$ has finite
  burden.
\end{conjecture}
Recall that, up to elementary equivalence, dp-minimal fields come in
three types:
\begin{enumerate}
\item \label{uno} Hahn series $F((T^\Gamma))$ where $F$ is a local
  field of characteristic 0 ($\Rr, \Qq_p$, or a finite extension), and
  where $\Gamma$ is a dp-minimal ordered abelian group.  Examples:
  \begin{equation*}
    \Rr, \Cc, \Qq_p, \Cc((T)), \Qq_3(\sqrt{-1}), \Qq_p((T))
  \end{equation*}
\item \label{dos} Hahn series $\Ff_p^{alg}((T^\Gamma))$, where
  $\Gamma$ is a $p$-divisible dp-minimal ordered abelian group.
\item \label{tres} The mixed characteristic analogue of (\ref{dos}).
\end{enumerate}
Conjecture~\ref{the-conj} seems likely when $K$ is of type
(\ref{uno}); for types (\ref{dos}-\ref{tres}) positive characteristic
may cause additional problems.

There may be some analogue of PRC and P$p$C for other dp-minimal
complete theories of fields.  One could then generalize
Conjecture~\ref{the-conj} to the bounded pseudo dp-minimal setting.

Also, in the dp-minimal case, Conjecture~\ref{the-conj} seems
plausible even when $K$ is a dp-minimal \emph{expansion} of a field,
because of the known compatibility between definable sets and the
canonical topology.

\subsection{Fields of finite dp-rank or finite burden}
Since we are stepping outside inp-minimality, we may as well conjecture
\begin{conjecture}\label{the-conj-2}
  If $K$ is a field of finite dp-rank, and $\Oo_1, \ldots, \Oo_n$ are
  $n$ valuation rings on $K$, then $(K,\Oo_1,\ldots,\Oo_n)$ has finite
  burden.
\end{conjecture}
This is probably intractable until dp-finite fields are classified.

Assuming one can complete the analogy
\begin{quote}
  real closed : pseudo real closed :: dp-finite : ?
\end{quote}
then one would also hope for an analogue of
Conjecture~\ref{the-conj-2} in the ``bounded pseudo dp-finite''
setting, though ``bounded'' needs to be changed (dp-finite fields
themselves need not be bounded!).

While we are here, we may as well make a very general conjecture:
\begin{conjecture}
  If $K$ is a field of finite burden, possibly with extra structure,
  and $\Oo$ is a valuation ring on $K$, then $(K,\Oo)$ has finite
  burden.
\end{conjecture}
There is no real approach to proving this, short of classifying the
fields of finite burden.

\subsection{Acknowledgments}
An earlier version of these results was presented at the conference
``Model Theory of Valued Fields'' at IHP in March 2018.  The author
would like to thank the organizers for the invitation, which provided
motivation to prove these new results.

The author would also like to thank Silvain Rideau-Kikuchi for helpful
discussions which provided encouragement to write up the results.

{\tiny This material is based upon work supported by the National
  Science Foundation under Awards No. DGE-1106400 and DMS-1803120.
  Any opinions, findings, and conclusions or recommendations expressed
  in this material are those of the author and do not necessarily
  reflect the views of the National Science Foundation.}

\bibliographystyle{plain} \bibliography{mybib}{}

\end{document}